\documentclass[11pt]{amsart}

\usepackage{enumerate}
\usepackage{amsmath,amsthm,verbatim,amssymb,amsfonts,amscd,amsopn,amsxtra,graphicx,lmodern}
\usepackage{hyperref}
\usepackage{tikz-cd} 
\usepackage{graphics}
\usepackage{mathrsfs}
\usepackage{relsize}
\usepackage{enumitem}
\usepackage[utf8]{inputenc}
\usepackage{csquotes}
\usepackage{amsthm}
\usepackage{thmtools}
\usepackage{mathtools}
\usepackage{microtype}
\usepackage{pdfrender,xcolor}
\usepackage[sort cites=true, backend=biber]{biblatex}
\usepackage{lua-visual-debug}
\usepackage{biblatex}

\tikzset{
	symbol/.style={
		draw=none,
		every to/.append style={
			edge node={node [sloped, allow upside down, auto=false]{$#1$}}}
	}
}

\begin{document}
	\pdfrender{StrokeColor=black,TextRenderingMode=2,LineWidth=0.2pt}	
	
	\title{Extensions of valuations to rational function fields over completions}

	\author{Arpan Dutta}

	\def\NZQ{\mathbb}               
	\def\NN{{\NZQ N}}
	\def\QQ{{\NZQ Q}}
	\def\ZZ{{\NZQ Z}}
	\def\RR{{\NZQ R}}
	\def\CC{{\NZQ C}}
	\def\AA{{\NZQ A}}
	\def\BB{{\NZQ B}}
	\def\PP{{\NZQ P}}
	\def\FF{{\NZQ F}}
	\def\GG{{\NZQ G}}
	\def\HH{{\NZQ H}}
	\def\UU{{\NZQ U}}
	\def\P{\mathcal P}
	
	%
	%
	\let\union=\cup
	\let\sect=\cap
	\let\dirsum=\oplus
	\let\tensor=\otimes
	\let\iso=\cong
	\let\Union=\bigcup
	\let\Sect=\bigcap
	\let\Dirsum=\bigoplus
	\let\Tensor=\bigotimes
	
	\theoremstyle{plain}
	\newtheorem{Theorem}{Theorem}[section]
	\newtheorem{Lemma}[Theorem]{Lemma}
	\newtheorem{Corollary}[Theorem]{Corollary}
	\newtheorem{Proposition}[Theorem]{Proposition}
	\newtheorem{Problem}[Theorem]{}
	\newtheorem{Conjecture}[Theorem]{Conjecture}
	\newtheorem{Question}[Theorem]{Question}
	
	\theoremstyle{definition}
	\newtheorem{Example}[Theorem]{Example}
	\newtheorem{Examples}[Theorem]{Examples}
	\newtheorem{Definition}[Theorem]{Definition}
	
	\theoremstyle{remark}
	\newtheorem{Remark}[Theorem]{Remark}
	\newtheorem{Remarks}[Theorem]{Remarks}
	
	\newcommand{\trdeg}{\mbox{\rm trdeg}\,}
	\newcommand{\rr}{\mbox{\rm rat rk}\,}
	\newcommand{\sep}{\mathrm{sep}}
	\newcommand{\ac}{\mathrm{ac}}
	\newcommand{\ins}{\mathrm{ins}}
	\newcommand{\res}{\mathrm{res}}
	\newcommand{\Gal}{\mathrm{Gal}\,}
	\newcommand{\ch}{\operatorname{char}}
	\newcommand{\Aut}{\mathrm{Aut}\,}
	\newcommand{\kras}{\mathrm{kras}\,}
	\newcommand{\dist}{\mathrm{dist}\,}
	\newcommand{\ord}{\mathrm{ord}\,}
	
	\newcommand{\n}{\par\noindent}
	\newcommand{\nn}{\par\vskip2pt\noindent}
	\newcommand{\sn}{\par\smallskip\noindent}
	\newcommand{\mn}{\par\medskip\noindent}
	\newcommand{\bn}{\par\bigskip\noindent}
	\newcommand{\pars}{\par\smallskip}
	\newcommand{\parm}{\par\medskip}
	\newcommand{\parb}{\par\bigskip}
	\let\epsilon\varepsilon
	\let\phi=\varphi
	\let\kappa=\varkappa
	
	\def \a {\alpha}
	\def \b {\beta}
	\def \s {\sigma}
	\def \d {\delta}
	\def \g {\gamma}
	\def \o {\omega}
	\def \l {\lambda}
	\def \th {\theta}
	\def \D {\Delta}
	\def \G {\Gamma}
	\def \O {\Omega}
	\def \L {\Lambda}
	%
	%
	\textwidth=15cm \textheight=22cm \topmargin=0.5cm
	\oddsidemargin=0.5cm \evensidemargin=0.5cm \pagestyle{plain}


	\address{Discipline of Mathematics, School of Basic Sciences, IIT Bhubaneswar, Argul,
		Odisha, India, 752050.}
	\email{arpandutta@iitbbs.ac.in}
	
	\date{\today}
	
	\thanks{}
	
	
	
	\maketitle

	
	\begin{abstract}
		Given a valued field $(K,v)$ and its completion $(\widehat{K},v)$, we study the set of all possible extensions of $v$ to $\widehat{K}(X)$. We show that any such extension is closely connected with the underlying subextension $(K(X)|K,v)$. The connections between these extensions are studied via minimal pairs, key polynomials, pseudo-Cauchy sequences and implicit constant fields. As a consequence, we obtain strong ramification theoretic properties of $(\widehat{K},v)$. We also give necessary and sufficient conditions for $(K(X),v)$ to be dense in $(\widehat{K}(X),v)$.
	\end{abstract}

	\tableofcontents
	
	
	\section{Introduction}
    
	Given a valued field $(K,v)$ and a rational function field $K(X)$ in one variable over $K$, it is a deep and classical problem to give a satisfactory characterization and description of all extensions of $v$ to $K(X)$. It is sufficient to understand the set of all extensions of $v$ to the polynomial ring $K[X]$, since the extended valuation necessarily satisfies $v(f/g) = vf-vg$ for all $f,g \in K[X]$. The typical example of such an extended valuation is the Gau{\ss} valuation: for a polynomial $f(X) := \sum_{i=0}^{n} c_i X^i \in K[X]$, we set $vf:= \min \{ vc_i \}$. Note that we have $vX=0$ in the Gau{\ss} valuation. Assigning an arbitrary value lying in some ordered abelian group containing $vK$ to $X$ and shifting the origin of the polynomial to an arbitrary element of $K$ gives the monomial valuations, which are prototypical examples of the so-called valuation transcendental extensions. Not every extension of $v$ to $K(X)$ is valuation transcendental; however, the rest can be realized as limits of a certain class of valuation transcendental extensions. A thorough understanding of this extension problem has given rise to a whole host of literature and several different inter-connected notions. The fundamental objects of study are minimal pairs of definition [\ref{AP sur une classe}, \ref{APZ all valns on K(X)}, \ref{APZ characterization of residual trans extns}, \ref{APZ2 minimal pairs}], key polynomials [\ref{MacLane key pols}, \ref{Vaquie key pols}, \ref{Nova Spiva key pol pseudo convergent}, \ref{Novacoski key poly and min pairs}], pseudo-Cauchy sequences [\ref{Ostrowski pcs}, \ref{Kaplansky}] and implicit constant fields [\ref{Kuh value groups residue fields rational fn fields}, \ref{Dutta min fields implicit const fields}, \ref{Dutta min pairs inertia ram deg impl const fields}, \ref{Dutta imp const fields key pols valn alg extns}]. We refer the reader to Section \ref{Sect prelims} for all the relevant definitions. We would also like to emphasise that the linked references offer but only a tiny glimpse into the vast body of work devoted to this problem.
	
	\pars In this article, we undertake a systematic study of the extension $(\widehat{K}(X)|\widehat{K},v)$, where $\widehat{K}$ denotes the completion of the valued field $(K,v)$. Our investigation shows that the extension $(\widehat{K}(X)|\widehat{K},v)$ is intimately connected with the underlying extension $(K(X)|K,v)$. The connection between these extensions is further explored in detail in terms of the aforementioned objects, viz. minimal pairs, key polynomials and implicit constant fields. Our first important result is the following:

	\begin{Theorem}\label{Thm unique common extn}
		Let $(K,v)$ be a valued field, $(\widehat{K},\widehat{v})$ be its completion and $w$ be an extension of $v$ to $K(X)$. Fix an extension $\overline{\widehat{v}}$ of $\widehat{v}$ to the algebraic closure $\overline{\widehat{K}}$ of $\widehat{K}$. Fix a common extension $\overline{w}$ of $w$ and $\overline{v}$ to $\overline{K}(X)$, where $\overline{v}:= \overline{\widehat{v}}|_{\overline{K}}$. Then there exists a common extension of $\overline{\widehat{v}}$ and $\overline{w}$ to $\overline{\widehat{K}}(X)$. Further, every such common extension restricts to the same valuation (up to identification of equivalent valuations) $\widehat{w}$ on $\widehat{K}(X)$.   
	\end{Theorem}
	It could be helpful to keep the following diagram in mind:
	\[  \begin{tikzcd}[row sep={50,between origins}, column sep={50,between origins}]
		& (K,v) \arrow[dl] \arrow[rr] \arrow[dd] & & (K(X),w) \arrow[dl] \arrow[dd] \\
		(\widehat{K},\widehat{v}) \arrow[rr, crossing over] \arrow[dd] & & (\widehat{K}(X),\widehat{w}) \\
		& (\overline{K},\overline{v}) \arrow[dl] \arrow[rr] & & (\overline{K}(X),\overline{w}) \arrow[dl] \\
		(\overline{\widehat{K}}, \overline{\widehat{v}}) \arrow[rr] & & (\overline{\widehat{K}}(X), \overline{\widehat{w}}) \arrow[from=uu, crossing over]\\
	\end{tikzcd}.  \]

	We will say that $(\widehat{K}(X)|\widehat{K},\widehat{w})$ is \textbf{induced by $(K(X)|K,w)$, $\overline{\widehat{v}}$ and $\overline{w}$}. When the valuations $\overline{\widehat{v}}$ and $\overline{w}$ are tacitly understood, we will simply say that $(\widehat{K}(X)|\widehat{K},\widehat{w})$ is induced by $(K(X)|K,w)$. The uniqueness of the assertion implies that we can talk of \emph{the} induced valuation once we fix the necessary extensions. In order to prove Theorem \ref{Thm unique common extn}, we first give a construction of the induced extension, defined in a unique manner. Then over the course of Proposition \ref{Prop completion valn tr induced}, Proposition \ref{Prop valn tr induced unique pair two cases} and Proposition \ref{Prop valn alg induced}, we will show that 
	\begin{itemize}
		\item given an extension $(\overline{\widehat{K}}(X)|K,v)$ of valued fields, the extension $(\widehat{K}(X)|\widehat{K},v)$ is induced by $(K(X)|K,v)$.
	\end{itemize} 
    During our preparation, we explore the strong connections of the induced extension $(\widehat{K}(X)|\widehat{K},\widehat{w})$ with the extension $(K(X)|K,w)$. Some of our primary findings are mentioned below:
	
	\subsection*{Case I: $(K(X)|K,w)$ is valuation transcendental} Then $(\widehat{K}(X)|\widehat{K},\widehat{w})$ is also valuation transcendental. Further, a minimal pair of definition for $w$ over $K$ also forms a minimal pair of definition for $\widehat{w}$ over $\widehat{K}$ (Corollary \ref{Coro CSKP valn tr}). We have the following subcases:
	
	\subsubsection*{Case IA: $vK$ is cofinal in $w K(X)$} In this case, a complete sequence of key polynomials for $w$ over $K$ forms a complete sequence of key polynomials for $\widehat{w}$ over $\widehat{K}$ (Theorem \ref{Thm CSKP valn tr cofinal}). Moreover, the extension $(\widehat{K}(X)|K(X),\widehat{w})$ is immediate (Corollary \ref{Coro CSKP valn tr}).
	
	\subsubsection*{Case IB: $(K(X)|K,w)$ admits a unique pair of definition $(a,\g)$} Take the minimal polynomials $Q_a$ and $\widehat{Q}_a$ of $a$ over $K$ and $\widehat{K}$ respectively. Then a complete sequence of key polynomials for $\widehat{w}$ over $\widehat{K}$ is obtained by taking an initial segment of a complete sequence of key polynomials for $w$ over $K$, and replacing $Q_a$ by $\widehat{Q}_a$ (Theorem \ref{Thm CSKP unique pair}).   
	
	\subsection*{Case II: $(K(X)|K,w)$ is valuation algebraic} There are again two subcases:
	
	\subsubsection*{Case IIA: $(\widehat{K}(X)|\widehat{K},\widehat{w})$ is valuation algebraic} In this case, a complete sequence of key polynomials for $w$ over $K$ forms a complete sequence of key polynomials for $\widehat{w}$ over $\widehat{K}$. Further, $(\widehat{K}(X)|K(X),\widehat{w})$ is immediate (Theorem \ref{Thm CSKP valn alg}). In this case, we will say that $(K(X)|K,w)$ is \textbf{valuation algebraic of type I}.
	
	\subsubsection*{Case IIB: $(\widehat{K}(X)|\widehat{K},\widehat{w})$ is valuation transcendental} In this case, $(\widehat{K}(X)|\widehat{K},\widehat{w})$ has a unique pair of definition $(a,\g)$ where $a\in\overline{\widehat{K}}\setminus\overline{K}$. A complete sequence of key polynomials for $\widehat{w}$ over $\widehat{K}$ is obtained by taking an initial segment of a complete sequence of key polynomials for $w$ over $K$, and adjoining $\widehat{Q}_a$. The element $a$ is \emph{uniquely} determined in the following way: take sequences $\{\g_\nu\}_{\nu<\l} \subseteq \overline{w}(X-\overline{K})$ and $\{a_\nu\}_{\nu<\l} \subseteq \overline{K}$ such that $\{\g_\nu\}_{\nu<\l}$ is cofinal in $\overline{w}(X-\overline{K})$ and $\overline{w}(X-a_\nu) = \g_\nu$ for all $\nu<\l$. Then $\{a_\nu\}_{\nu<\l}$ is a Cauchy sequence in $(\overline{K},\overline{v})$ and has $a$ as the unique limit in $\overline{\widehat{K}}$ (Theorem \ref{Thm CSKP valn alg}). In this case, we will say that $(K(X)|K,w)$ is \textbf{valuation algebraic of type II}.

	\pars Observe that the induced valuation $(\widehat{K}(X)|\widehat{K},\widehat{w})$ is determined uniquely once we fix $\overline{\widehat{v}}$ and a common extension $\overline{w}$ of $w$ and $\overline{v}$ to $\overline{K}(X)$. It is a natural question to inquire about the connections between two induced valuations corresponding to the choice of two distinct common extensions. This has been explored in the next result:
	
	\begin{Theorem}\label{Thm connection between distinct induced valns}
		Let notations and assumptions be as in Theorem \ref{Thm unique common extn}. Take two common extension $\overline{w}$ and $\overline{w}_1$ of $w$ and $\overline{v}$ to $\overline{K}(X)$. Consider the corresponding induced extensions $(\widehat{K}(X)|\widehat{K},\widehat{w})$ and $(\widehat{K}(X)|\widehat{K},\widehat{w}_1)$. Then,
		\sn (i) $(\widehat{K}(X)|\widehat{K},\widehat{w})$ is valuation algebraic if and only if $(\widehat{K}(X)|\widehat{K},\widehat{w}_1)$ is valuation algebraic,
		\n (ii) $(\widehat{K}(X)|\widehat{K},\widehat{w})$ is valuation transcendental if and only if $(\widehat{K}(X)|\widehat{K},\widehat{w}_1)$ is valuation transcendental.\\
		When the induced extensions are valuation transcendental, then we can choose corresponding minimal pairs of definition $(a,\g)$ and $(a_1,\g)$ over $\widehat{K}$ such that $a$ and $a_1$ are conjugates over $\widehat{K}$. Further, if $a, a_1 \in \overline{K}$, then they are also conjugates over $K$.
	\end{Theorem}
	
	\pars We can employ our findings to obtain strong ramification theoretic properties of the completion $\widehat{K}$:

	\begin{Theorem}\label{Thm K^c properties}
		Let $(K,v)$ be a henselian valued field and fix an extension of $v$ to $\overline{\widehat{K}}$. Then the following statements hold true:
		\sn (i) $(K,v)$ is separably defectless $\Longleftrightarrow$ $(\widehat{K},v)$ is separably defectless $\Longleftrightarrow$ $(\widehat{K},v)$ is defectless. 
		\n (ii) $(K,v)$ is separably tame $\Longleftrightarrow$ $(\widehat{K},v)$ is separably tame $\Longleftrightarrow$ $(\widehat{K},v)$ is tame.
		\n (iii) $(K,v)$ is separable-algebraically maximal $\Longleftrightarrow$ $(\widehat{K},v)$ is separable-algebraically maximal.
		\n (iv) $K$ is separable-algebraically closed $\Longleftrightarrow$ $\widehat{K}$ is separable-algebraically closed $\Longleftrightarrow$ $\widehat{K}$ is algebraically closed.
		\n (v) $(K,v)$ is algebraically maximal $\Longrightarrow$ $(\widehat{K},v)$ is algebraically maximal. However, the converse is not true.
	\end{Theorem}

	\pars Given an algebraic extension $(L|K,v)$ of valued fields, it is well-known that $\widehat{K}$ is contained in $\widehat{L}$. However, this containment fails to hold when the  extension is transcendental. If this containment was to hold true in the simplest possible case, that is when $L=K(X)$, then we would necessarily have the chain $K(X) \subseteq \widehat{K}(X) \subseteq \widehat{K(X)}$. As a consequence, $K(X)$ would lie dense in $\widehat{K}(X)$. In Section \ref{Sect density} we obtain a necessary and sufficient condition for this to hold. We first illustrate the failure of $\widehat{K}$ to be contained in $\widehat{K(X)}$ whenever $(K(X)|K,v)$ is either value transcendental with a unique pair of definition or valuation algebraic of type II (Proposition \ref{Prop not density of K(X)}). We then obtain the following result in the complimentary case:  
	
	\begin{Theorem} \label{Thm density of K(X)}
		Let $(\overline{\widehat{K}}(X)|K,v)$ be an extension of valued fields. Assume that either $(K(X)|K,v)$ is valuation transcendental such that $vK$ is cofinal in $vK(X)$, or that $(K(X)|K,v)$ is valuation algebraic of type I. Take $\a\in v\widehat{K}(X)$ and polynomials $f,g \in \widehat{K}[X]$. Then there exist polynomials $f^\prime, g^\prime \in K[X]$ such that the following conditions are satisfied:
		\begin{align*}
			\deg f = \deg f^\prime, \, \, \deg g = \deg g^\prime,& \, \, vf  = vf^\prime, \, \, vg  = v g^\prime,\\
			v(f-f^\prime) > \a,  \, \,  v(g-g^\prime) > \a, & \, \, v(\frac{f}{g} - \frac{f^\prime}{g^\prime})  > \a.
		\end{align*}
	\end{Theorem}
	Note that we actually obtain the stronger result that the polynomial ring $K[X]$ lies dense in the polynomial ring $\widehat{K}[X]$ under the assumptions of Theorem \ref{Thm density of K(X)}.

	\pars Finally, given an extension of valued fields $(\overline{\widehat{K}(X)}|K,v)$, we study the connections between the implicit constant fields $IC(\widehat{K}(X)|\widehat{K},v)$ and $IC(K(X)|K,v)$. The implicit constant field of the extension $(K(X)|K,v)$ is defined to be the relative algebraic closure of $K$ in the henselization $K(X)^h$ of $K(X)$. It follows directly from the definition that $\widehat{K}^h.IC(K(X)|K,v) \subseteq IC(\widehat{K}(X)|\widehat{K},v)$, where $\widehat{K}^h$ denotes the henselization of $\widehat{K}$. We obtain the following result in Section \ref{Sect imp cnst fields}: 
	\begin{Proposition}\label{Prop IC of K and K hat}
		Let $(\overline{\widehat{K}(X)}|K,v)$ be an extension of valued fields. Assume that $(K(X)|K,v)$ is not valuation algebraic of type II. Then,
		\[ \widehat{K}^h. IC(K(X)|K,v) = IC(\widehat{K}(X)|\widehat{K},v).  \]
	\end{Proposition}
	As an immediate corollary, we obtain that $IC(\widehat{K}(X)|\widehat{K},v)$ is the completion of $IC(K(X)|K,v)$ provided $(K,v)$ is henselian, $IC(K(X)|K,v)|K$ is a finite extension and $(K(X)|K,v)$ is not valuation algebraic of type II.
	
	
	\section*{Acknowledgements}
	
	Preliminary versions of the work was done when the author was appointed as a postdoctoral fellow (supported by the Post-Doctoral Fellowship of the National Board of Higher Mathematics, India) in IISER Mohali. The author is presently supported by the Seed Grant F.28-3(15)/2021-22/SP114 from IIT Bhubaneswar. He is deeply indebted to Wael Mahboub for multiple inspiring discussions and insightful suggestions.


\section{Preliminaries}\label{Sect prelims}

Throughout this article, we denote the algebraic closure and the separable-algebraic closure of a field $K$ by $\overline{K}$ and $K^\sep$. If $L$ and $K$ are subfields of a larger field $\O$, then the \textbf{compositum} of $L$ and $K$ will be denoted by $L.K$. A valued field $(K,v)$ is a field $K$ equipped with a valuation $v$. An extension of valued fields $(L|K,v)$ is an extension of fields $L|K$, $v$ is a valuation on $L$ and $K$ is endowed with the restricted valuation. The value group $v(K^\times)$ of $(K,v)$ will be denoted by $vK$ and the residue field by $Kv$. The value of an element $a\in K$ will be denoted by $va$ and its residue by $av$.

\subsection{Density and Ramification theory} Let $(L|K,v)$ be an extension of valued fields. We will say that $(K,v)$ is \textbf{dense} in $(L,v)$ if for every element $a\in L$ and for every $\g\in vL$, there exists some $b\in K$ such that $v(a-b) >\g$. If this holds, then the extension $(L|K,v)$ is \textbf{immediate}, that is, $vL=vK$ and $Lv=Kv$. The maximal extension in which $(K,v)$ is dense is said to be the \textbf{completion} of $(K,v)$, which is unique up to valuation preserving isomorphism. We denote the completion of $K$ by $\widehat{K}$. It is well-known that when $(L|K,v)$ is algebraic, then $\widehat{K}\subseteq \widehat{L}$. Moreover, if $(L|K,v)$ is a finite extension of valued fields, then $\widehat{L} = L.\widehat{K}$.

\pars A valued field $(K,v)$ is said to be \textbf{henselian} if the valuation $v$ admits a unique extension to $\overline{K}$. Every valued field admits a \textbf{henselization}, that is, a minimal algebraic extension which is henselian. The henselization is uniquely determined once we fix an extension of $v$ to $\overline{K}$. Hence we can talk of \emph{the} henselization of $(K,v)$ which we will denote by $K^h$. The extension $(K^h|K,v)$ is an immediate, separable-algebraic extension. If $(L|K,v)$ is an algebraic extension, then $L^h = L.K^h$. In particular, every algebraic extension of a henselian valued field is again henselian.  

\pars The fact that $(K^h|K,v)$ is an algebraic extension implies that $\widehat{K}$ is contained in the completion of $K^h$. The completion of a henselian field is again henselian [\ref{Warner topo fields}, Theorem 32.19]. As a consequence, we have the following chain of inclusions:
\[  K^h\subseteq \widehat{K}^h\subseteq \widehat{K^h}, \]  
where $\widehat{K}^h$ denotes the henselization of $\widehat{K}$, and $\widehat{K^h}$ denotes the completion of $K^h$.

\pars Let $(L|K,v)$ be a finite extension of valued fields. Further, assume that $v$ admits a unique extension to $L$. Then the \textbf{Lemma of Ostrowski} gives us that
\[  [L:K] = (vL:vK)[Lv:Kv]p^d, \]
where $d\in\NN$ and $p$ is the \textbf{characteristic exponent} of $(K,v)$, that is, $p=\ch Kv$ whenever $\ch Kv >0$, and $p=1$ otherwise. The integer $p^d$ is said to be the \textbf{defect} of the extension $(L|K,v)$ and is denoted by $d(L|K,v)$. If $p^d=1$, then we say that the extension is \textbf{defectless}. 

\pars Observe that we can always apply the Lemma of Ostrowski to extensions of henselian valued fields. An algebraic extension $(L|K,v)$ of henselian valued fields will be said to be defectless if every finite subextension is defectless. A henselian field $(K,v)$ will be said to be (separably) defectless if the extension $(\overline{K}|K,v)$ is a defectless extension (similarly for $(K^\sep|K,v)$). An arbitrary valued field $(K,v)$ will be said to be (separably) defectless if so is its henselization.

\pars An algebraic extension $(L|K,v)$ of henselian valued fields is said to be \textbf{tame} if every finite subextension $(E|K,v)$ satisfies the following conditions:
\sn (TE1) $\ch Kv$ does not divide $(vE:v K)$,
\n (TE2) the residue field extension $Ev|Kv$ is separable,
\n (TE3) $(E|K,v)$ is defectless.\\
A henselian valued field admits a unique maximal tame extension [\ref{Kuh vln model}, Theorem 11.1]. It is called the \textbf{absolute ramification field of $(K,v)$} and we will denote it by $K^r$.

\pars A henselian field $(K,v)$ is said to be a (separably) tame field if every (separable-algebraic) algebraic extension $(L|K,v)$ is a tame extension. Equivalently, $(K,v)$ is tame field if and only if $\overline{K}=K^r$ and $(K,v)$ is a separably tame field if and only if $K^\sep = K^r$.

\pars A valued field is said to be \textbf{maximal} if it does not admit any nontrivial immediate extension. If it does not admit any nontrivial immediate algebraic extension, it is said to be \textbf{algebraically maximal}. It has been shown by Krull [\ref{Krull}] that any formal power series field $k((t))$ equipped with the $t$-adic valuation is a maximal valued field. 


\subsection{Krasner's Lemma} Let $(K,v)$ be a valued field and fix an extension of $v$ to $\overline{K}$. Take $a\in\overline{K}\setminus K$ which is not purely inseparable over $K$. The \textbf{Krasner constant} of $a$ over $K$ is defined as
\[ \kras(a,K) := \max \{ v(\s a - \tau a) \mid \s, \tau \in \Gal(\overline{K}|K) \text{ and } \s a \neq \tau a \}. \]
When $v$ admits a unique extension from $K$ to $K(a)$, we observe that
\[ \kras(a,K) := \max \{ v(\s a - a) \mid \s\in \Gal(\overline{K}|K) \text{ and } \s a \neq a \}. \]
The following is the important Krasner's Lemma:
\begin{Lemma}
	Let $(K,v)$ be a henselian valued field. Assume that $a,b\in \overline{K}$ such that $v(a-b) > \kras(a,K)$. Then $K(a,b)|K(b)$ is a purely inseparable extension. In particular, if $a$ is separable-algebraic over $K$, then $a\in K(b)$. 
\end{Lemma}


\subsection{Pseudo-Cauchy sequences} A well-ordered set $\{z_{\nu}\}_{\nu<\l}$ in a valued field $(K,v)$, where $\l$ is a limit ordinal, is said to be a \textbf{pseudo-Cauchy sequence} (pCs) if $v(z_{\nu_1} - z_{\nu_2}) < v(z_{\nu_2} - z_{\nu_3})$ for all $\nu_1<\nu_2<\nu_3<\l$. It follows from the triangle inequality that $v(z_\nu - z_{\rho}) = v(z_\nu-z_{\nu+1})$ for all $\nu<\rho<\l$. Set $\g_\nu := v(z_\nu-z_{\nu+1})$. An element $y$ in some valued field extension $(L,v)$ of $(K,v)$ is said to be a \textbf{limit} of $\{  z_\nu\}_{\nu<\l}$ if $v(y -z_\nu) = \g_\nu$ for all $\nu<\l$. Note that a pCs need not have a unique limit. Indeed, it follows from the triangle inequality that for $y,y^\prime \in L$, 
\[  y \text{ and } y^\prime \text{ are limits of } \{z_{\nu}\}_{\nu<\l} \text{ if and only if } v(y-y^\prime) > \g_\nu \text{ for all } \nu<\l. \]
A pCs $\{z_{\nu}\}_{\nu<\l}$ in $(K,v)$ is said to be a \textbf{Cauchy sequence} if $\{\g_\nu\}_{\nu<\l}$ is cofinal in $vK$. A Cauchy sequence in $(K,v)$ admits a unique limit in the completion $\widehat{K}$.

\pars Take a polynomial $f(x)\in K[x]$ and a pCs $\{z_\nu\}_{\nu<\l}$ over $K$. It follows from [\ref{Kaplansky}, Lemma 5] that the sequence $\{v f(z_\nu)\}_{\nu<\l}$ is either ultimately constant, or it is ultimately monotonically increasing. If $\{v f(z_\nu)\}_{\nu<\l}$ is ultimately constant for all polynomials $f$ over $K$, then $\{z_{\nu}\}_{\nu<\l}$ is said to be a pCs of \textbf{transcendental type}. Otherwise, it is of \textbf{algebraic type}. 


\subsection{Extensions of $v$ from $K$ to $K(X)$}
The extension $(K(X)|K,v)$ satisfies the famous Abhyankar inequality: 
\[ \rr vK(X)/vK + \trdeg [K(X)v:Kv] \leq 1, \]
where $\rr vK(X)/vK$ is the $\QQ$-dimension of the divisible hull $(vK(X)/vK) \tensor_{\ZZ} \QQ$. This inequality is a consequence of [\ref{Bourbaki}, Chapter VI, \S 10.3, Theorem 1]. The extension $(K(X)|K,v)$ is said to be \textbf{value transcendental} if $\rr vK(X)/vK = 1$ and \textbf{residue transcendental} if $\trdeg[K(X)v:Kv] = 1$. The extension is said to be \textbf{valuation transcendental} if it is either value or residue transcendental. Otherwise, it is said to be \textbf{valuation algebraic}. It follows from [\ref{Kuh value groups residue fields rational fn fields}, Lemma 3.3] that $(K(X)|K,v)$ is value (residue) transcendental if and only if $(L(X)|L,v)$ is also value (residue) transcendental, where $L$ is an arbitrary algebraic extension of $K$. Further, $(K(X)|K,v)$ is valuation algebraic if and only if $(\overline{K}(X)|\overline{K},v)$ is immediate.

\pars Take a polynomial $f(X) = \sum_{i=0}^{n} c_i X^i\in K[X]$ and an element $a\in K$. We have a unique expansion $f(X) = \sum_{i=0}^{n} C_i(X-a)^i$, where the coefficients $C_i$ are given by the \textbf{Hasse derivative}:
\[ C_i = c_i + \binom{i+1}{i} a c_{i+1} + \dotsc + \binom{n}{i} a^{n-i} c_n. \]
Take $\g$ in some ordered abelian group containing $vK$. We define the map $v_{a,\g} : K[X] \longrightarrow vK + \ZZ\g$ by setting
\[ v_{a,\g}f(X) := \min \{ v C_i + i\g \}, \]
and extend $v_{a,\g}$ canonically to $K(X)$. Then $v_{a,\g}$ is a valuation transcendental extension of $v$ from $K$ to $K(X)$ [\ref{Kuh value groups residue fields rational fn fields}, Lemma 3.10]. By definition, $v_{a,\g}(X-b) = \min \{ v_{a,\g}(X-a),v(a-b) \} \leq v_{a,\g}(X-a)$ for all $b\in K$. It follows that $v_{a,\g}(X-a) = \g = \max v_{a,\g} (X-K)$, where
\[ v_{a,\g}(X-K):= \{ v_{a,\g}(X-c) \mid c\in K \}. \]
Observe that $v_{0,0}$ is nothing but the Gau{\ss} valuation.


\subsection{Minimal pairs}
Let $(\overline{K}(X)|K,v)$ be an extension of valued fields. It has been observed in [\ref{Kuh value groups residue fields rational fn fields}, Theorem 3.11] that $(K(X)|K,v)$ is valuation transcendental if and only if $v=v_{a,\g}$ for some $a\in\overline{K}$ and $\g\in v\overline{K}(X)$. Moreover, $(K(X)|K,v)$ is value transcendental if and only if $\g\notin v\overline{K}$, that is, if and only if $\g$ is not a torsion element modulo $vK$. Such a pair $(a,\g)\in\overline{K}\times v\overline{K}(X)$ is said to be a \textbf{pair of definition for $v$ over $K$}. A valuation transcendental extension can admit multiple pairs of definition. Indeed, given a pair of definition $(a,\g)$ for $v$ over $K$, it has been observed in [\ref{AP sur une classe}, Proposition 3] that $(b,\g)\in\overline{K}\times v\overline{K}(X)$ forms another pair of definition if and only if
\[  v(a-b) \geq \g. \] 
A pair of definition $(a,\g)$ is said to be a \textbf{minimal pair of definition for $v$ over $K$} if
\[  v(a-b)\geq\g\Longrightarrow [K(a):K]\leq [K(b):K] \text{ for all } b\in\overline{K}. \] 
It then follows from the well-ordering principle that a valuation transcendental extension always admits a minimal pair of definition. 

\pars Let $(K(X)|K,v)$ be a valuation transcendental extension. Fix an extension of $v$ to $\overline{K}(X)$ and take a minimal pair of definition $(a,\g)$ for $v$ over $K$. Then the following cases are possible [\ref{Dutta min fields implicit const fields}, Theorem 3.7]:
\sn (i) $\g> v\overline{K}$. Then $v$ is value transcendental and $(a,\g)$ is the unique pair of definition for $v$ over $K$.
\n (ii) $vK$ is cofinal in $vK(X)$. Then there exists a minimal pair of definition $(b,\g)$ for $v$ over $K$ such that $b$ is separable-algebraic over $K$.\\
Now assume that $(a,\g)$ is the unique pair of definition for $v$ over $K$. Take any $f(X)\in K[X]$ and write $f(X) = z\prod_{i=1}^{n} (X-z_i)$ where $z,z_i \in\overline{K}$. Then $vf = vz + \sum_{i=1}^{n} v(X-z_i)$. By definition, $v(X-z_i) = \min\{ \g, v(a-z_i) \}$. The fact that $\g>v\overline{K}$ implies that $v(X-z_i) \in v\overline{K}$ whenever $a\neq z_i$. It follows that $vf\in v\overline{K}$ whenever $f(a)\neq 0$. Otherwise $vf>v\overline{K}$.

\subsection{Key polynomials}

Let $(\overline{K}(X)|K,v)$ be an extension of valued fields. For a polynomial $f(X)\in K[X]$, we define
\[ \d(f):= \max\{ v(X-a)\mid a \text{ is a root of }f \}. \]
A root $a$ of $f$ such that $\d(f) = v(X-a)$ is said to be a \textbf{maximal root of $f$}. A monic polynomial $Q(X)\in K[X]$ is said to be a \textbf{key polynomial for $v$ over $K$} if 
\[ \deg f < \deg Q \Longrightarrow \d(f) < \d(Q) \text{ for all } f(X)\in K[X]. \]
Given a monic polynomial $Q(X)\in K[X]$ and an arbitrary polynomial $f(X)\in K[X]$, we have a unique expansion $f = \sum_{i=0}^{m} f_i Q^i$, where $f_i(X)\in K[X]$ with $0\leq \deg f_i < \deg Q$. Consider the map
\[ v_Q : K[X]\longrightarrow vK(X) \]
given by $v_Q(f):= \min\{ vf_i + ivQ \}$. Extend $v_Q$ canonically to $K(X)$. A sufficient condition for $v_Q$ to be a valuation on $K(X)$ is that $Q$ be a key polynomial for $v$ over $K$ [\ref{Nova Spiva key pol pseudo convergent}, Proposition 2.6], but it is not a necessary condition [\ref{Novacoski key poly and min pairs}, Proposition 2.3]. A family $\L$ of key polynomials for $v$ over $K$ is said to form a \textbf{complete sequence of key polynomials for $v$ over $K$} if it satisfies the following conditions:
\sn (CSKP1) $\d(Q)\neq\d(Q^\prime)$ for $Q,Q^\prime\in\L$ with $Q\neq Q^\prime$,
\n (CSKP2) $\L$ is well ordered with respect to the ordering given by $Q<Q^\prime$ if and only if $\d(Q)<\d(Q^\prime)$,
\n (CSKP3) for any $f(X)\in K[X]$, there exists some $Q\in\L$ such that $\deg Q \leq \deg f$ and $vf = v_Qf$.  

\begin{Lemma}\label{Lemma vf leq v (a,g) f}
	Assume that $(\overline{K}(X)|K,v)$ is an extension of valued fields. Take $a\in\overline{K}$ and set $\g:= v(X-a)$. Take $f(X)\in K[X]$. Then,
	\begin{align*}
		vf &> v_{a,\g}f \text{ if and only if } \d(f)> \g,\\
		vf &= v_{a,\g}f \text{ if and only if } \d(f)\leq \g.
	\end{align*}
\end{Lemma}

\begin{proof}
	Write $f(X) = z \prod_{i=1}^{n} (X-z_i)$ where $z,z_i \in\overline{K}$. By definition, $v_{a,\g} (X-z_i) = \min \{ \g, v(a - z_i) \}$. It then follows from the triangle inequality that $v_{a,\g}(X-z_i) \leq v(X-z_i)$. As a consequence, 
	\[ v_{a,\g}f \leq vf. \]
	Further, $v_{a,\g}f < vf$ if and only if $v_{a,\g} (X-z_i) < v(X-z_i)$ for some $i$, which holds if and only if $v(X-z_i) > \g = v(a-z_i)$. The lemma now follows.
\end{proof}

\begin{Lemma}\label{Lemma vQ key pol}
	Assume that $(\overline{K}(X)|K,v)$ is an extension of valued fields. Take a key polynomial $Q(X)$ for $v$ over $K$ and a maximal root $a$ of $Q$. Then,
	\[ v_{a, \d(Q)} |_{K(X)} = v_Q. \]
\end{Lemma}

\begin{proof}
	By definition, $(a,\d(Q))$ is a pair of definition for $v_{a,\d(Q)}$ over $K$. Take another pair of definition $(a^\prime, \d(Q))$. Then $v(a-a^\prime) \geq \d(Q)=v(X-a)$ and hence $v(X-a^\prime) \geq \d(Q)$. Take the minimal polynomial $Q^\prime(X)$ of $a^\prime $ over $K$. The fact that $v(X-a^\prime) \geq \d(Q)$ implies that $\d(Q^\prime) \geq \d(Q)$. As a consequence of $Q(X)$ being a key polynomial for $v$ over $K$, we conclude that $\deg Q^\prime \geq \deg Q$, that is, $[K(a^\prime):K] \geq [K(a):K]$. It follows that $(a,\d(Q))$ is a minimal pair of definition for $v_{a,\d(Q)}$ over $K$. In light of [\ref{Novacoski key poly and min pairs}, Theorem 1.1], we conclude that $v_{a, \d(Q)} |_{K(X)} = v_Q$.
\end{proof}

The following result is now immediate in light of Lemma \ref{Lemma vf leq v (a,g) f} and Lemma \ref{Lemma vQ key pol}:

\begin{Proposition}\label{Prop vf = v_Q f iff d(f) leq d(Q)}
	Assume that $(\overline{K}(X)|K,v)$ is an extension of valued fields. Take a polynomial $f(X)\in K[X]$ and a key polynomial $Q(X)$ for $v$ over $K$. Then,
	\[ vf = v_Q f \text{ if and only if } \d(f) \leq \d(Q). \]
\end{Proposition}

As an immediate consequence, we observe that a complete sequence of key polynomials $\{Q_\nu\}_{\nu\in\L}$ for the extension $(K(X)|K,v)$ necessarily satisfies that
\[  \{ \d(Q_\nu) \}_{\nu\in\L} \text{ is a cofinal subset of } v(X-\overline{K}). \]

\begin{Lemma}\label{Lemma CSKP valn tr final element}
	Assume that $(\overline{K}(X)|K,v)$ is an extension of valued fields such that $(K(X)|K,v)$ is valuation transcendental. Take a minimal pair of definition $(a,\g)$ for $v$ over $K$ and the minimal polynomial $Q_a$ of $a$ over $K$. Then there exists a complete sequence of key polynomials for $v$ over $K$ with $Q_a$ as its final element.
\end{Lemma}

\begin{proof}
	The fact that $(a,\g)$ is a pair of definition implies that $\g = \max v(X-\overline{K})$. Take a complete sequence of key polynomials $\{Q_\nu\}_{\nu\in\L}$ for $v$ over $K$. Then $\{\d(Q_\nu)\}_{\nu<\in\L}$ is cofinal in $v(X-\overline{K})$. As a consequence, $\{\d(Q_\nu)\}_{\nu\in\L}$ contains $\g$ as its final element. Take the corresponding final key polynomial $Q_b$ with the maximal root $b$. It follows from [\ref{Novacoski key poly and min pairs}, Theorem 1.1] that $(b,\g)$ is a minimal pair of definition for $v$ over $K$. The minimality of $(a,\g)$ then implies that $\deg Q_a = \deg Q_b$. Write the complete sequence of key polynomials as $\{Q_\nu\}_{\nu\in\L_0}\union\{Q_b\}$. We can then directly check that $\{Q_\nu\}_{\nu\in\L_0}\union\{Q_a\}$ also forms a complete sequence of key polynomials for $v$ over $K$.
\end{proof}

\subsection{Continuity of roots}

We refer the reader to [\ref{Kuh book}] for a proof of the following important theorem:

\begin{Theorem}[Continuity of roots]
	Let $(K,v)$ be an arbitrary valued field. Fix an extension of $v$ to $\overline{K}$. Take the Gau{\ss} valuation $v_{0,0}$ extending $v$ to $\overline{K}(X)$. Take $\a\in vK$ and $f,f^\prime\in K[X]$. Let $\deg f = \deg f^\prime = n$ and take the leading coefficients $c_n$ and $c^\prime_n$ of $f$ and $f^\prime$ respectively. If 
	\[ v_{0,0}(f-f^\prime) > n^n\a -3n^n (v_{0,0} f-vc_n) + vc_n, \]
	then we can write $f(X) = c_n \prod_{i=1}^{n} (X-z_i)$ and $f^\prime(X) = c^\prime_n \prod_{i=1}^{n} (X-z^\prime_i)$ such that $v(z_i - z^\prime_i) >\a$ for all $i$.
\end{Theorem}

 As an immediate corollary, we obtain the following result which plays an important role throughout the rest of this article:
 
\begin{Lemma}\label{Lemma deg f and d(f)}
	Take an extension of valued fields $(\overline{\widehat{K}}(X)|K,v)$ and a polynomial $f(X)\in \widehat{K}[X]$. Assume that $\d(f) < \a$ for some $\a\in vK$. Then there exists $f^\prime(X) \in K[X]$ such that $\deg f = \deg f^\prime$, $vf = v f^\prime$ and $\d(f) = \d(f^\prime)$.
\end{Lemma}

\begin{proof}
	Write $f(X) = \sum_{i=0}^{n} c_i X^i = c_n \prod_{i=1}^{n} (X-z_i)$. By the continuity of roots, there exists $\d\in vK$ such that after suitable renumbering,
	\[ v(z_i - z^\prime_i) > \a \text{ for all } i \text{ whenever } v(c_i - c^\prime_i) > \d \text{ for all }i, \]
	where $c^\prime_n \neq 0$ and $\sum_{i=0}^{n} c^\prime_i X^i = c^\prime_n\prod_{i=1}^{n} (X-z^\prime_i)$. Choose suitable $c^\prime_i \in K$ and write $f^\prime(X) := \sum_{i=0}^{n} c^\prime_i X^i \in K[X]$. Then $\deg f = \deg f^\prime$. Further, 
	\[ v(z_i - z^\prime_i) > \a > \d(f) \geq v(X-z_i) \text{ for all }i. \]
	It then follows from the triangle inequality that $v(X-z_i) = v(X-z^\prime_i)$ for all $i$. As a consequence, we obtain that $\d(f) = \d(f^\prime)$. Moreover, observe that we can always choose $c^\prime_n \in K$ satisfying $v(c_n-c^\prime_n) > \max \{ \d,vc_n \}$. Hence $vc_n = vc^\prime_n$ and as a consequence,
	\[ v f^\prime = v c^\prime_n + \sum_{i=1}^{n} v(X-z^\prime_i) = v c_n + \sum_{i=1}^{n} v(X-z_i) = vf.  \]
\end{proof}


\section{Valuation transcendental case}\label{Sect valn tr}

	Take a valued field $(K,v)$ and a valuation transcendental extension $w$ of $v$ to $K(X)$. Fix an extension $\overline{\widehat{v}}$ of $v$ to $\overline{\widehat{K}}$. Set $\widehat{v}:= \overline{\widehat{v}}|_{\widehat{K}}$ and $\overline{v}:= \overline{\widehat{v}}|_{\overline{K}}$. Take a common extension $\overline{w}$ of $w$ and $\overline{v}$ to $\overline{K}(X)$ and a pair of definition $(a,\g)$ for $w$ over $K$. Consider the extension $\overline{\widehat{w}}:= \overline{\widehat{v}}_{a,\g}$ of $\overline{\widehat{v}}$ to $\overline{\widehat{K}}(X)$. The fact that $\overline{\widehat{v}}$ extends $\overline{v}$ implies that $\overline{\widehat{w}}$ extends $\overline{w}$. Set $\widehat{w}:= \overline{\widehat{w}}|_{\widehat{K}(X)}$. Then $\widehat{w}$ is an extension of $w$ to $\widehat{K}(X)$. We will say that the extension $(\widehat{K}(X)|\widehat{K}, \widehat{w})$ is \textbf{induced} by $(K(X)|K,w)$. 
	
	\pars If $(b,\g)$ is another pair of definition for $w$ over $K$, then $\overline{v}(a-b) \geq \g$. Consequently, $\overline{\widehat{v}}_{b,\g} = \overline{\widehat{v}}_{a,\g} = \overline{\widehat{w}}$. Thus the induced extension $(\widehat{K}(X)|\widehat{K}, \widehat{w})$ is uniquely determined. For ease of notation, from now on we will denote all the valuations by the same letter $v$.
	
	\pars Let $(\overline{K}(X)|K,v)$ be an extension of valued fields and assume that $(K(X)|K,v)$ is valuation transcendental. Fix an extension of $v$ to $\overline{\widehat{K}}$ and consider the induced extension $(\widehat{K}(X)|\widehat{K},v)$. Take a pair of definition $(a,\g)$ for $v$ over $K$. The fact that the completion is an immediate extension implies that
	\[ v\overline{K} = \QQ\tensor_{\ZZ} vK = \QQ\tensor_{\ZZ} v\widehat{K} = v\overline{\widehat{K}}. \]  
	Recall that $(K(X)|K,v)$ is residue transcendental if and only if $\g\in v\overline{K}$. It follows that the induced extension $(\widehat{K}(X)|\widehat{K}, v)$ is residue (value) transcendental if and only if $(K(X)|K,v)$ is residue (value) transcendental. Further, $vK$ is cofinal in $vK(X)$ if and only if $(K(X)|K,v)$ does not admit a unique pair of definition. It follows that $v\widehat{K}$ is cofinal in $v\widehat{K}(X)$ if and only if $vK$ is cofinal in $vK(X)$.

	\subsection{$vK$ is cofinal in $vK(X)$}
	
	\begin{Theorem}\label{Thm CSKP valn tr cofinal}
		Let $(\overline{K}(X)|K,v)$ be an extension of valued fields. Assume that $(K(X)|K,v)$ is valuation transcendental and that $vK$ is cofinal in $vK(X)$. Fix an extension of $v$ to $\overline{\widehat{K}}$. Take a complete sequence of key polynomials $\{Q_\nu\}_{\nu\in\L}$ for $v$ over $K$. Then $\{Q_\nu\}_{\nu\in\L}$ forms a complete sequence of key polynomials for the induced extension $(\widehat{K}(X)|\widehat{K},v)$.
	\end{Theorem}
	
	\begin{proof}
		Take any $\nu\in\L$ and suppose that $Q_\nu$ is not a key polynomial for $v$ over $\widehat{K}$. Then there exists some polynomial $f(X)\in \widehat{K}[X]$ such that $\deg f < \deg Q_\nu$ and $\d(f) \geq \d(Q_\nu)$. The assumption that $vK$ is cofinal in $vK(X)$ implies that $v\widehat{K}$ is cofinal in $v\widehat{K}(X)$. Consequently, there exists some $\a\in vK$ such that $\d(f) < \a$. Applying Lemma \ref{Lemma deg f and d(f)}, we obtain a polynomial $f^\prime(X) \in K[X]$ such that $\deg f = \deg f^\prime$ and $\d(f) = \d(f^\prime)$. It follows that $\deg f^\prime < \deg Q_\nu$ and $\d(f^\prime) \geq \d(Q_\nu)$. However, this contradicts the fact that $Q_\nu$ is a key polynomial for $v$ over $K$. We have thus obtained that $Q_\nu$ is a key polynomial for $v$ over $\widehat{K}$ for all $\nu\in\L$.
		
		\pars Take any $g(X)\in \widehat{K}[X]$. It follows from Lemma \ref{Lemma deg f and d(f)} that there exists some polynomial $g^\prime(X) \in K[X]$ such that $\d(g) = \d(g^\prime)$ and $\deg g = \deg g^\prime$. The fact that $\{Q_\nu\}_{\nu\in\L}$ forms a complete sequence of key polynomials for $v$ over $K$ implies that $vg^\prime = v_{Q_\nu}g^\prime$ for some $\nu\in\L$ with $\deg Q_\nu \leq \deg g^\prime$. In light of Proposition \ref{Prop vf = v_Q f iff d(f) leq d(Q)}, we obtain that $\d(g^\prime) \leq \d(Q_\nu)$. As a consequence, $\d(g) \leq \d(Q_\nu)$. The observation that $Q_\nu$ is a key polynomial for $v$ over $\widehat{K}$ now yields that $vg = v_{Q_\nu}g$. 
	\end{proof}
	
	\begin{Corollary}\label{Coro CSKP valn tr}
		Let notations and assumptions be as in Theorem \ref{Thm CSKP valn tr cofinal}. Take a minimal pair of definition $(a,\g)$ for $v$ over $K$. Then $(a,\g)$ is also a minimal pair of definition for the induced extension $(\widehat{K}(X)|\widehat{K},v)$. Further, 
		\[ [K(a):K] = [\widehat{K}(a):\widehat{K}] \text{ and } (\widehat{K}(X)|K(X),v) \text{ is immediate}. \]
	\end{Corollary}
	
	\begin{proof}
		In light of Lemma \ref{Lemma CSKP valn tr final element}, we take a complete sequence of key polynomials $\{Q_\nu\}_{\nu\in\L} \union \{Q_a\}$ for $v$ over $K$ with $Q_a$ as the final element, where $Q_a$ is the minimal polynomial of $a$ over $K$. It follows from Theorem \ref{Thm CSKP valn tr cofinal} that $\{Q_\nu\}_{\nu\in\L} \union \{Q_a\}$ also forms a complete sequence of key polynomials for the induced extension $(\widehat{K}(X)|\widehat{K},v)$ with $Q_a$ as the final element. Observe that $a$ is a maximal root of $Q_a$ and $\d(Q_a) = \g$. It then follows from [\ref{Novacoski key poly and min pairs}, Theorem 1.1] that $(a,\g)$ is a minimal pair of definition for $v$ over $\widehat{K}$. Moreover, $Q_a$ is a key polynomial for $v$ over $\widehat{K}$ and hence irreducible. As a consequence, 
		\[ [K(a):K] = \deg Q_a = [\widehat{K}(a):\widehat{K}]. \]
		The fact that $K(a)|K$ is a finite extension implies that $\widehat{K(a)} = \widehat{K}(a)$. Completion being an immediate extension implies that 
		\[ vK(a) = v\widehat{K}(a) \text{ and } K(a)v = \widehat{K}(a)v. \]
		The assertion that $(\widehat{K}(X)|K(X),v)$ is an immediate extension now follows from [\ref{Dutta min fields implicit const fields}, Remark 3.1] and [\ref{Dutta min fields implicit const fields}, Remark 3.3].   
	\end{proof}
	
	Let $(\overline{\widehat{K}}(X)|K,v)$ be an extension of valued fields. Assume that $(\widehat{K}(X)|\widehat{K},v)$ is valuation transcendental and take a pair of definition $(a,\g)$ for $v$ over $\widehat{K}$. Then $a\in \overline{\widehat{K}}$ by definition. Observe that the inclusion $K\subseteq \widehat{K}$ implies that $\overline{K} \subseteq \overline{\widehat{K}}$. Further, $\overline{K}|K$ being an algebraic extension implies that $\widehat{K} \subseteq \widehat{\overline{K}}$. The fact that the completion of an algebraically closed field is again algebraically closed implies that $\widehat{\overline{K}}$ is algebraically closed. We can now conclude from the inclusion $\widehat{K} \subseteq \widehat{\overline{K}}$ that 
	\[ \overline{K} \subseteq \overline{\widehat{K}} \subseteq \widehat{\overline{K}}. \]  
	It follows that $a\in\widehat{\overline{K}}$. Assume that $v\widehat{K}$ is cofinal in $v\widehat{K}(X)$. Then there exists some $\a\in v\widehat{K}$ such that $\g<\a$. Consequently, there exists some $a^\prime\in\overline{K}$ such that $v(a^\prime-a) > \a > \g$. Thus $(a^\prime,\g)$ is also a pair of definition for $v$ over $\widehat{K}$. The fact that $a^\prime\in\overline{K}$ implies that $(a^\prime,\g)$ is also a pair of definition for $v$ over $K$. It follows that $(\widehat{K}(X)|\widehat{K},v)$ is induced by the extension $(K(X)|K,v)$. Further, $vK$ is cofinal in $vK(X)$. It then follows from [\ref{Dutta min fields implicit const fields}, Proposition 3.6] that we can always choose a minimal pair of definition $(b,\g)$ for $v$ over $K$ such that $b\in K^\sep$. Applying Corollary \ref{Coro CSKP valn tr}, we thus obtain the following result:
	
	\begin{Proposition}\label{Prop completion valn tr induced}
		Let $(\overline{\widehat{K}}(X)|K,v)$ be an extension of valued fields. Assume that $(\widehat{K}(X)|\widehat{K})$ is valuation transcendental such that $v\widehat{K}$ is cofinal in $v\widehat{K}(X)$. Then $(\widehat{K}(X)|\widehat{K},v)$ is induced by the valuation transcendental extension $(K(X)|K,v)$. In this case, we can choose a minimal pair of definition $(a,\g)$ for $v$ over $K$ such that $a\in K^\sep$. Further, $(a,\g)$ is also a minimal pair of definition for $v$ over $\widehat{K}$, $[K(a):K] = [\widehat{K}(a):\widehat{K}]$ and $(\widehat{K}(X)|K(X),v)$ is immediate.
	\end{Proposition}
	
	\subsection{$(K(X)|K,v)$ admits a unique pair of definition} We now consider the complementary case, that is, $(K(X)|K,v)$ is a value transcendental extension with a unique pair of definition $(a,\g)$. The uniqueness of $(a,\g)$ implies that any complete sequence of key polynomials for $v$ over $K$ contains the minimal polynomial $Q_a$ of $a$ over $K$ as its final element (Lemma \ref{Lemma CSKP valn tr final element}). 
	
	\begin{Theorem}\label{Thm CSKP unique pair}
		 Let $(\overline{K}(X)|K,v)$ be an extension of valued fields. Assume that $(K(X)|K,v)$ is value transcendental with a unique pair of definition $(a,\g)$. Fix an extension of $v$ to $\overline{\widehat{K}}$. Take a complete sequence of key polynomials $\{Q_\nu\}_{\nu\in\L}\union \{Q_a\}$ for $v$ over $K$, where $Q_a$ is the minimal polynomial of $a$ over $K$. Then $\{Q_\nu \}_{\nu\in\L^\prime} \union \{ \widehat{Q}_a \}$ forms a complete sequence of key polynomials for the induced extension $(\widehat{K}(X)|\widehat{K},v)$, where $\widehat{Q}_a$ is the minimal polynomial of $a$ over $\widehat{K}$ and $\L^\prime:= \{ \nu\in\L \mid \deg Q_\nu \leq \deg \widehat{Q}_a \}$.
	\end{Theorem}
	
	\begin{proof}
		Observe that $\g> v\overline{K} = v\overline{\widehat{K}}$. It follows that $(a,\g)$ is the unique pair of definition for $v$ over $\widehat{K}$. Consequently, any complete sequence of key polynomials for $v$ over $\widehat{K}$ contains $\widehat{Q}_a$ as its final element. Moreover, for any $f(X)\in\widehat{K}[X]$ such that $f(a)\neq 0$, we have that $\d(f) \in v\overline{K}$ and hence
		\[ \d(f) < \a  \text{ for some } \a\in vK.\]
		Take some $\nu\in\L^\prime$ and suppose that $Q_\nu$ is not a key polynomial for $v$ over $\widehat{K}$. Then there exists $f(X)\in \widehat{K}[X]$ with $\d(f) \geq \d(Q_\nu)$ and $\deg f < \deg Q_\nu \leq \deg \widehat{Q}_a$. It follows from our preceding discussions that $\d(f) < \a$ for some $\a\in vK$. Applying Lemma \ref{Lemma deg f and d(f)}, we obtain some $f^\prime(X) \in K[X]$ such that $\deg f^\prime = \deg f < \deg Q_\nu$ and $\d(f^\prime) = \d(f) \geq \d(Q_\nu)$, thereby contradicting the fact that $Q_\nu$ is a key polynomial for $v$ over $K$. It follows that $Q_\nu$ is a key polynomial for $v$ over $\widehat{K}$ for all $\nu\in\L^\prime$. 
		\pars It follows from Proposition \ref{Prop vf = v_Q f iff d(f) leq d(Q)} that $vg = v_{\widehat{Q}_a}g$ for all $g(X)\in\widehat{K}[X]$. Assume that $\deg g < \deg \widehat{Q}_a$. Then $\d(g) \in v\overline{K}$. In light of Lemma \ref{Lemma deg f and d(f)}, there exists some $g^\prime(X)\in K[X]$ such that $\deg g = \deg g^\prime$ and $\d(g) = \d(g^\prime)$. The fact that $\{Q_\nu\}_{\nu\in\L}\union\{Q_a\}$ forms a complete sequence of key polynomials for $v$ over $K$ implies that $vg^\prime = v_{Q_\nu} g^\prime$ for some $\nu\in\L$ with $\deg Q_\nu\leq \deg g^\prime$. As a consequence, $\deg Q_\nu \leq \deg g < \deg \widehat{Q}_a$ and hence $\nu\in\L^\prime$. Thus $Q_\nu$ is a key polynomial for $v$ over $\widehat{K}$. Further, it follows from Proposition \ref{Prop vf = v_Q f iff d(f) leq d(Q)} that $\d(g^\prime) \leq \d(Q_\nu)$. Consequently, we obtain that $\d(g)\leq \d(Q_\nu)$ and hence $vg = v_{Q_\nu}g$. We have thus proved the theorem. 
	\end{proof}

	A direct analogue of Proposition \ref{Prop completion valn tr induced} does not hold in this case. Indeed, assume that $(\overline{\widehat{K}}(X)|K,v)$ is an extension of valued fields such that $(\widehat{K}(X)|\widehat{K},v)$ is value transcendental with a unique pair of definition $(a,\g)$ over $\widehat{K}$. If $a\in\overline{K}$, then $(K(X)|K,v)$ is also value transcendental with $(a,\g)$ as the unique pair of definition and as a consequence, $(\widehat{K}(X)|\widehat{K},v)$ is induced by $(K(X)|K,v)$. Otherwise, $a\in\overline{\widehat{K}} \setminus \overline{K}$. Take any $z\in\overline{K}$. Then $v(X-z) = v(a-z) \in v\overline{K}$. Consequently, $vK(X) \subseteq v\overline{K}$ and hence $vK(X)/vK$ is a torsion group. Further, $(\widehat{K}(X)|\widehat{K},v)$ being a value transcendental extension implies that $\widehat{K}(X)v|\widehat{K}v$ is an algebraic extension. It follows that $K(X)v|Kv$ is an algebraic extension. We have thus obtained the following result:
	
	\begin{Proposition}\label{Prop valn tr induced unique pair}
		Assume that $(\overline{\widehat{K}}(X)|K,v)$ is an extension of valued fields such that $(\widehat{K}(X)|\widehat{K},v)$ is value transcendental with a unique pair of definition $(a,\g)$ over $\widehat{K}$. Then the following cases are possible:
		\sn (i) $a\in\overline{K}$. Then $(\widehat{K}(X)|\widehat{K},v)$ is induced by the value transcendental extension $(K(X)|K,v)$.
		\n (ii) $a\in\overline{\widehat{K}} \setminus \overline{K}$. Then $(K(X)|K,v)$ is valuation algebraic.
	\end{Proposition}


\section{Valuation algebraic case}\label{Sect valn alg}

Take a valued field $(K,v)$ and a valuation algebraic extension $w$ of $v$ to $K(X)$. Fix an extension $\overline{\widehat{v}}$ of $v$ to $\overline{\widehat{K}}$. Set $\widehat{v}:= \overline{\widehat{v}}|_{\widehat{K}}$ and $\overline{v}:= \overline{\widehat{v}}|_{\overline{K}}$. Take a common extension $\overline{w}$ of $w$ and $\overline{v}$ to $\overline{K}(X)$. Then $(\overline{K}(X)|\overline{K}, \overline{w})$ is an immediate extension [\ref{Kuh value groups residue fields rational fn fields}, Lemma 3.3]. As a consequence, the set $\overline{w}(X-\overline{K})$ does not contain a maximal element. Take sequences $\{\g_\nu\}_{\nu<\l} \subseteq \overline{w}(X-\overline{K})$ and $\{a_\nu\}_{\nu<\l} \subseteq \overline{K}$, where $\l$ is a limit ordinal, satisfying the following conditions:
\begin{align*}\label{(C1)-(C2)}
&\{\g_\nu\}_{\nu<\l} \text{ is a well-ordered cofinal subset of } \overline{w}(X-\overline{K})\tag{C1},\\
& \overline{w}(X-a_\nu) = \g_\nu \text{ for all }\nu<\l \tag{C2}.
\end{align*}
Observe that $\{a_\nu\}_{\nu<\l}$ is a pCs in $(\overline{K},\overline{v})$ with limit $X$. Suppose that $\{a_\nu\}_{\nu<\l}$ is of algebraic type. Then there exists some $l\in\overline{K}$ such that $l$ is also a limit of the sequence [\ref{Kaplansky}, Theorem 3]. Consequently, $\overline{w}(X-l) > \g_\nu$ for all $\nu<\l$, which contradicts the cofinality of $\{\g_\nu\}_{\nu<\l}$. Hence 
\[ \{a_\nu\}_{\nu<\l} \text{ is a pCs of transcendental type in } (\overline{K}, \overline{v}). \]
By definition, $\{a_\nu\}_{\nu<\l}$ is also a pCs in $(\overline{\widehat{K}}, \overline{\widehat{v}})$. We now consider the two distinct possibilities separately. 

\subsection*{Case A: $\{a_\nu\}_{\nu<\l}$ is a pCs of algebraic type in $(\overline{\widehat{K}}, \overline{\widehat{v}})$} In this case, there exists some $a\in\overline{\widehat{K}}$ which is a limit of the sequence $\{a_\nu\}_{\nu<\l}$. We first observe that $\{a_\nu\}_{\nu<\l}$ is necessarily a Cauchy sequence in $(\overline{K},\overline{v})$. Suppose that this is not the case. Then there exists some $\d\in\overline{v}\overline{K}$ such that $\g_\nu<\d$ for all $\nu<\l$. The fact that $a\in\overline{\widehat{K}}\subseteq\widehat{\overline{K}}$ implies that there exists $b\in\overline{K}$ such that $\overline{\widehat{v}}(a-b) \geq \d$. Consequently, $\overline{\widehat{v}}(a-b) > \g_\nu$ for all $\nu<\l$ and hence $b$ is a limit of $\{a_\nu\}_{\nu<\l}$, thereby contradicting the fact that $\{a_\nu\}_{\nu<\l}$ is a pCs of transcendental type in $(\overline{K},\overline{v})$.

\pars Take the ordered abelian group $(\ZZ\dirsum\overline{v}\overline{K})_{\text{lex}}$ equipped with the lexicographic order. Set $\G:= (\ZZ\dirsum\overline{v}\overline{K})_{\text{lex}}$ and embed $\overline{v}\overline{K}$ into $\G$ by setting $\a \mapsto (0,\a)$. Take $\g:= (1,0)$ and the valuation $\overline{\widehat{w}}:= \overline{\widehat{v}}_{a,\g}$ extending $\overline{\widehat{v}}$ to $\overline{\widehat{K}}(X)$. Set $\widehat{w}:= \overline{\widehat{w}}|_{\widehat{K}(X)}$. Take any $z\in\overline{K}$. By definition, $\overline{\widehat{w}}(X-z) = \overline{\widehat{v}}(a-z)$. The fact that $z$ is not a limit of $\{a_\nu\}_{\nu<\l}$ implies that the sequence $\{ \overline{v}(a_\nu -z) \}_{\nu<\l}$ is ultimately constant. Denote the ultimately constant value of the sequence as $\b$. Then there exists some $\nu<\l$ such that $\b< \g_\nu$. From the triangle inequality, it now follows that
\[ \overline{\widehat{w}}(X-z) = \overline{\widehat{w}}(X-a_\nu + a_\nu - z)  = \b. \] 
On the other hand, the fact that $\{a_\nu\}_{\nu<\l}$ is of transcendental type in $(\overline{K},\overline{v})$ with $X$ as a limit implies that $\overline{w}(X-z) = \b$ [\ref{Kaplansky}, Theorem 2]. It follows that $\overline{\widehat{w}}|_{\overline{K}(X)} = \overline{w}$ and consequently $\widehat{w}|_{K(X)} = w$.

\pars Take a complete sequence of key polynomials $\{Q_\mu\}_{\mu\in\L}$ for $w$ over $K$. Employing the same arguments as in the proof of Theorem \ref{Thm CSKP unique pair}, we observe that $\{Q_\mu\}_{\mu\in\L^\prime} \union \{\widehat{Q}_a\}$ forms a complete sequence of key polynomials for $\widehat{w}$ over $\widehat{K}$, where $\widehat{Q}_a$ is the minimal polynomial of $a$ over $\widehat{K}$ and $\L^\prime:= \{ \mu\in \L \mid \deg Q_\mu \leq \deg\widehat{Q}_a \}$.

\subsection*{Case B: $\{a_\nu\}_{\nu<\l}$ is a pCs of transcendental type in $(\overline{\widehat{K}}, \overline{\widehat{v}})$} In this case, there exists a unique extension $\overline{\widehat{w}}$ of $\overline{\widehat{v}}$ to $\overline{\widehat{K}}(X)$ such that $X$ is a limit of $\{a_\nu\}_{\nu<\l}$; moreover, the extension is immediate [\ref{Kaplansky}, Theorem 2]. It follows directly that $\overline{\widehat{w}}|_{\overline{K}(X)} = \overline{w}$. Set $\widehat{w}:= \overline{\widehat{w}}|_{\widehat{K}(X)}$. The fact that $(\overline{\widehat{K}}(X)|\overline{\widehat{K}}, \overline{\widehat{w}})$ is an immediate extension implies that $(\widehat{K}(X)|\widehat{K},\widehat{w})$ is valuation algebraic. As a consequence, $\d(f) \in \overline{v}\overline{K}$ for all $f\in\widehat{K}[X]$. Employing Lemma \ref{Lemma deg f and d(f)} and the same arguments as in the proof of Theorem \ref{Thm CSKP valn tr cofinal}, we conclude that any complete sequence of key polynomials for $w$ over $K$ also forms a complete sequence of key polynomials for $\widehat{w}$ over $\widehat{K}$.

\pars It follows from our preceding discussions that $\{a_\nu\}_{\nu<\l}$ is a Cauchy sequence in $(\overline{K},\overline{v})$ whenever $\{a_\nu\}_{\nu<\l}$ is a pCs of algebraic type in $(\overline{\widehat{K}}, \overline{\widehat{v}})$. As a consequence, $\{a_\nu\}_{\nu<\l}$ is a pCs of transcendental type in $(\overline{\widehat{K}}, \overline{\widehat{v}})$ if and only if the same holds for any other pair of sequences $\{\g^\prime_{\nu^\prime}\}_{\nu^\prime<\l^\prime} \subseteq \overline{w}(X-\overline{K})$ and $\{a^\prime_{\nu^\prime}\}_{\nu^\prime< \l^\prime} \subseteq \overline{K}$ satisfying the conditions $(C1)$ and $(C2)$. Applying the construction in [\ref{Dutta imp const fields key pols valn alg extns}, Section 5], without any loss of generality, we can now assume that $\{a_\nu\}_{\nu<\l}$ and $\{\g_\nu\}_{\nu<\l}$ satisfy the following additional condition:
\[  (a_\nu,\g_\nu) \text{ is a minimal pair of definition for $\overline{v}_{a_\nu,\g_\nu}$ over }K. \tag{C3} \]
Observe that $(\widehat{K}(X)|\widehat{K}, \overline{\widehat{v}}_{a_\nu,\g_\nu})$ is induced by $(K(X)|K,\overline{v}_{a_\nu,\g_\nu})$. The fact that $\g_\nu \in \overline{v}\overline{K}$ implies that $(a_\nu,\g_\nu)$ is not a unique pair of definition. As a consequence, it follows from Corollary \ref{Coro CSKP valn tr} that
\begin{align*}
	 \overline{v}_{a_\nu,\g_\nu} K(X) &= \overline{\widehat{v}}_{a_\nu,\g_\nu} \widehat{K}(X),\\
	K(X)\overline{v}_{a_\nu,\g_\nu} &= \widehat{K}(X)\overline{\widehat{v}}_{a_\nu,\g_\nu},\\
	(a_\nu,\g_\nu) \text{ is a minimal }&\text{pair of definition for }  \overline{\widehat{v}}_{a_\nu,\g_\nu} \text{ over } \widehat{K}.
\end{align*}
Further, $\{\g_\nu\}_{\nu<\l}$ is cofinal in $\overline{\widehat{w}}(X-\overline{\widehat{K}})$. Indeed, take any $z\in\overline{\widehat{K}}$. The fact that $(\overline{\widehat{K}}(X)|\overline{\widehat{K}}, \overline{\widehat{w}})$ is immediate implies that $\overline{\widehat{w}}(X-z) \in \overline{\widehat{v}}\overline{\widehat{K}}=\overline{v}\overline{K}$. Take some $z^\prime\in\overline{K}$ such that $\overline{\widehat{v}}(z-z^\prime) > \overline{\widehat{w}}(X-z)$. It follows from the triangle inequality that $\overline{\widehat{w}}(X-z) = \overline{w}(X-z^\prime)$. Thus $\overline{\widehat{w}}(X-\overline{\widehat{K}}) = \overline{w}(X-\overline{K})$ and hence $\{\g_\nu\}_{\nu<\l}$ is cofinal in $\overline{\widehat{w}}(X-\overline{\widehat{K}})$. In light of [\ref{APZ all valns on K(X)}, Theorem 5.1], we now obtain that
\begin{align*}
	wK(X) = \Union_{\nu<\l} \overline{v}_{a_\nu,\g_\nu} K(X) &= \Union_{\nu<\l} \overline{\widehat{v}}_{a_\nu,\g_\nu} \widehat{K}(X) = \widehat{w}\widehat{K}(X),\\
	K(X)w = \Union_{\nu<\l} K(X) \overline{v}_{a_\nu,\g_\nu} &= \Union_{\nu<\l} \widehat{K}(X) \overline{\widehat{v}}_{a_\nu,\g_\nu} = \widehat{K}(X)\widehat{w}.
\end{align*}

\begin{Remark}
	Contrary to the valuation transcendental case, it is not immediately clear why an extension $(\widehat{K}(X)|\widehat{K},\widehat{w})$ induced by a valuation algebraic extension $(K(X)|K,w)$ is defined in a unique manner. To look into this matter more closely, we take distinct sequences $\{\g^\prime_{\nu^\prime}\}_{\nu^\prime<\l^\prime} \subseteq \overline{w}(X-\overline{K})$ and $\{a^\prime_{\nu^\prime}\}_{\nu^\prime< \l^\prime} \subseteq \overline{K}$ satisfying the conditions $(C1)$ and $(C2)$. Observe that $\{a_\nu\}_{\nu<\l}$ is of transcendental type in $(\overline{\widehat{K}},\overline{\widehat{v}})$ if and only if the same holds for $\{a^\prime_{\nu^\prime}\}_{\nu^\prime< \l^\prime}$. The valuation $\overline{\widehat{w}}$ is defined uniquely in this scenario by [\ref{Kaplansky}, Theorem 2].
	
	\pars We now assume that $\{a_\nu\}_{\nu<\l}$ is a pCs of algebraic type in $(\overline{\widehat{K}},\overline{\widehat{v}})$. Take the limit $a\in\overline{\widehat{K}}$ of $\{a_\nu\}_{\nu<\l}$. By construction, the sequence $\{a^\prime_{\nu^\prime}\}_{\nu^\prime< \l^\prime}$ admits $X$ as a limit. The fact that $\overline{\widehat{w}}(X-a) > \overline{v}\overline{K}$ implies that $\overline{\widehat{w}}(X-a) > \g^\prime_{\nu^\prime}$ for all $\nu^\prime< \l^\prime$. As a consequence, $a$ is also a limit of $\{a^\prime_{\nu^\prime}\}_{\nu^\prime< \l^\prime}$. The valuation $\overline{\widehat{w}}$ is thus uniquely defined. 
\end{Remark}

We gather our findings in the next theorem. For ease of notation, we will again denote all the valuations by the same letter $v$.

\begin{Theorem}\label{Thm CSKP valn alg}
	Let $(\overline{K}(X)|K,v)$ be an extension of valued fields. Assume that $(K(X)|K,v)$ is valuation algebraic. Fix an extension of $v$ to $\overline{\widehat{K}}$. Take a cofinal well-ordered subset $\{\g_\nu\}_{\nu<\l} \subseteq v(X-\overline{K})$ and $\{a_\nu\}_{\nu<\l}\subseteq \overline{K}$ such that $v(X-a_\nu) = \g_\nu$ for all $\nu<\l$. Take a complete sequence of key polynomials $\{Q_\mu\}_{\mu\in\L}$ for $v$ over $K$. Then the following cases are possible:
	\sn (Case I.) $\{a_\nu\}_{\nu<\l}$ is a pCs of transcendental type in $(\overline{\widehat{K}},v)$. Then $\{Q_\mu\}_{\mu\in\L}$ forms a complete sequence of key polynomials for the induced extension $(\widehat{K}(X)|\widehat{K},v)$. Further, $(\widehat{K}(X)|K(X),v)$ is an immediate extension.
	\n (Case II.) $\{a_\nu\}_{\nu<\l}$ is a pCs of algebraic type in $(\overline{\widehat{K}},v)$. Then $\{a_\nu\}_{\nu<\l}$ is a Cauchy sequence in $(\overline{K},v)$. Take the unique limit $a\in\overline{\widehat{K}}$ of $\{a_\nu\}_{\nu<\l}$. Then $\{Q_\mu\}_{\mu\in\L^\prime} \union \{\widehat{Q}_a\}$ forms a complete sequence of key polynomials for for the induced extension $(\widehat{K}(X)|\widehat{K},v)$, where $\widehat{Q}_a$ is the minimal polynomial of $a$ over $\widehat{K}$ and $\L^\prime:= \{\mu\in\L \mid \deg Q_\mu \leq \deg \widehat{Q}_a\}$. Further, the element $a$ is uniquely determined, that is, it is independent of the choice of the sequences $\{\g_\nu\}_{\nu<\l}$ and $\{a_\nu\}_{\nu<\l}$. 
\end{Theorem}

The extension $(K(X)|K,v)$ will be said to be \textbf{valuation algebraic of type I} if $\{a_\nu\}_{\nu<\l}$ is a pCs of transcendental type in $(\overline{\widehat{K}},v)$. Otherwise, we will say that $(K(X)|K,v)$ is \textbf{valuation algebraic of type II}.

\pars We now investigate whether we can modify Proposition \ref{Prop valn tr induced unique pair}. Assume that $(\overline{\widehat{K}}(X)|K,v)$ is an extension of valued fields such that $(\widehat{K}(X)|\widehat{K},v)$ is value transcendental with a unique pair of definition $(a,\g)$. It follows from [\ref{Dutta min fields implicit const fields}, Remark 3.3] that $v\overline{\widehat{K}}(X) = \ZZ\g\dirsum v\overline{K}$. Further, the fact that $(a,\g)$ is the unique pair of definition implies that $\g>v\overline{K}$. We now make use of the following general fact: let $\ZZ\g\dirsum G$ be an ordered abelian group where $\g> G$. Take the ordered abelian group $(\ZZ\dirsum G)_{\text{lex}}$ and embed $G$ in $(\ZZ\dirsum G)_{\text{lex}}$ by setting $\a \mapsto (0,\a)$ for all $\a\in G$. Then,
\begin{align*}
	\phi: \ZZ\g\dirsum G &\longrightarrow (\ZZ\dirsum G)_{\text{lex}},\\
	\g &\mapsto (1,0),\\
	\a &\mapsto (0,\a) \text{ for all }\a\in G,
\end{align*}
is an order preserving isomorphism. Recall that two valuations $v_1$ and $v_2$ on a field $K$ are said to be \textbf{equivalent} if there is an order preserving isomorphism $\phi: v_1 K \iso v_2 K$ such that $v_2 = \phi\circ v_1$. 

\pars After identifying equivalent valuations, we can thus assume that $v\overline{\widehat{K}}(X) = (\ZZ\dirsum v\overline{K})_{\text{lex}}$, where $\g = (1,0)$ and $\a \mapsto (0,\a)$ for all $\a\in v\overline{K}$. Further, assume that $a\in\overline{\widehat{K}}\setminus\overline{K}$. In light of Proposition \ref{Prop valn tr induced unique pair}, $(K(X)|K,v)$ is valuation algebraic and hence $(\overline{K}(X)|\overline{K},v)$ is immediate. Take sequences $\{\g_\nu\}_{\nu<\l}\subseteq v(X-\overline{K})$ and $\{a_\nu\}_{\nu<\l} \subseteq \overline{K}$ satisfying conditions $(C1)$ and $(C2)$. $X$ is a limit of $\{a_\nu\}_{\nu<\l}$ by construction. The fact that $v(X-a)=\g>\g_\nu$ for all $\nu<\l$ implies that $a$ is also a limit of $\{a_\nu\}_{\nu<\l}$. We have thus obtained the following result:

\begin{Proposition}\label{Prop valn tr induced unique pair two cases}
	Assume that $(\overline{\widehat{K}}(X)|K,v)$ is an extension of valued fields such that $(\widehat{K}(X)|\widehat{K},v)$ is valuation transcendental with a unique pair of definition $(a,\g)$. After identifying equivalent valuations, assume that $v\overline{\widehat{K}}(X) = (\ZZ\dirsum v\overline{K})_{\text{lex}}$, where $\g = (1,0)$ and $\a \mapsto (0,\a)$ for all $\a\in v\overline{K}$. Then the following cases are possible:
	\sn (i) $a\in\overline{K}$. Then $(\widehat{K}(X)|\widehat{K},v)$ is induced by the valuation transcendental extension $(K(X)|K,v)$.
	\n (ii) $a\in\overline{\widehat{K}}\setminus\overline{K}$. Then $(\widehat{K}(X)|\widehat{K},v)$ is induced by the valuation algebraic extension $(K(X)|K,v)$. Take a cofinal well-ordered subset $\{\g_\nu\}_{\nu<\l}$ of $v(X-\overline{K})$ and $\{a_\nu\}_{\nu<\l} \subseteq \overline{K}$ such that $v(X-a_\nu) = \g_\nu$ for all $\nu<\l$. Then $\{a_\nu\}_{\nu<\l}$ is a Cauchy sequence in $(\overline{K},v)$ with $a$ as a limit.
\end{Proposition}

It is left to check whether any valuation algebraic extension $(\widehat{K}(X)|\widehat{K},v)$ is again induced by the subextension $(K(X)|K,v)$.

\begin{Proposition}\label{Prop valn alg induced}
	Assume that $(\overline{\widehat{K}}(X)|K,v)$ is an extension of valued fields such that $(\widehat{K}(X)|\widehat{K},v)$ is valuation algebraic. Then $(\widehat{K}(X)|\widehat{K},v)$ is induced by the valuation algebraic extension $(K(X)|K,v)$. Moreover, the extension $(\widehat{K}(X)|K(X),v)$ is immediate.
\end{Proposition}

\begin{proof}
	Consider the chain $v\widehat{K} = vK \subseteq vK(X) \subseteq v\widehat{K}(X)$. The fact that $v\widehat{K}(X)/v\widehat{K}$ is a torsion group implies that $vK(X)/vK$ is a torsion group. Similarly, $K(X)v|Kv$ is an algebraic extension. Take sequences $\{\g_\nu\}_{\nu<\l}\subseteq v(X-\overline{K})$ and $\{a_\nu\}_{\nu<\l}\subseteq \overline{K}$ satisfying the conditions $(C1)$ and $(C2)$. Then $\{a_\nu\}_{\nu<\l}$ is a pCs of transcendental type in $(\overline{K},v)$ with limit $X$. Suppose that $\{a_\nu\}_{\nu<\l}$ is of algebraic type in $(\overline{\widehat{K}},v)$ and take the limit $a\in\overline{\widehat{K}}$. Set $\g:= v(X-a)$. Then $\g>\g_\nu$ for all $\nu<\l$. The fact that $(\overline{\widehat{K}}(X)|\overline{\widehat{K}},v)$ is an immediate extension implies that $\g\in v\overline{\widehat{K}} = v\overline{K}$. Consequently, we can choose $a^\prime\in\overline{K}$ such that $v(a-a^\prime) \geq \g$, whence $a^\prime$ is a limit of $\{a_\nu\}_{\nu<\l}$, thereby yielding a contradiction. It follows that $\{a_\nu\}_{\nu<\l}$ is a pCs of transcendental type in $(\overline{\widehat{K}},v)$. Since $X$ is a limit of $\{a_\nu\}_{\nu<\l}$, we obtain that $(\widehat{K}(X)|\widehat{K},v)$ is induced by $(K(X)|K,v)$ from the uniqueness assertion of [\ref{Kaplansky}, Theorem 2]. The final assertion of the proposition now follows from Theorem \ref{Thm CSKP valn alg}.
\end{proof}


\section{Proofs of Theorem \ref{Thm unique common extn} and Theorem \ref{Thm connection between distinct induced valns}}
 
 We first give a \textbf{proof of Theorem \ref{Thm unique common extn}}:
\begin{proof}
	Let $\overline{\widehat{w}}$ and $\overline{\widehat{w}}_1$ be two distinct common extensions of $\overline{w}$ and $\overline{\widehat{v}}$ to $\overline{\widehat{K}}(X)$. Set $\widehat{w}_1:= \overline{\widehat{w}_1}|_{\widehat{K}(X)}$ and $\widehat{w}:= \overline{\widehat{w}}|_{\widehat{K}(X)}$. In light of Proposition \ref{Prop completion valn tr induced}, Proposition \ref{Prop valn tr induced unique pair two cases} and Proposition \ref{Prop valn alg induced}, we obtain that both $(\widehat{K}(X)|\widehat{K},\widehat{w})$ and $(\widehat{K}(X)|\widehat{K},\widehat{w}_1)$ are induced by $(K(X)|K,w)$, $\overline{\widehat{v}}$ and $\overline{w}$, up to identification of equivalent valuations. The fact that the induced valuation is defined in a unique manner then implies that $\widehat{w}_1 = \widehat{w}$. 
\end{proof}

We now give a \textbf{proof of Theorem \ref{Thm connection between distinct induced valns}}:

\begin{proof}
	Fix an extension $\widehat{\overline{v}}$ of $\overline{\widehat{v}}$ to $\widehat{\overline{K}}$. The fact that $\overline{w}$ and $\overline{w}_1$ are extensions of $w$ to $\overline{K}(X)$ implies that there exists some $\overline{\s} \in \Gal (\overline{K}(X)|K(X))$ such that $\overline{w} = \overline{w}_1 \circ \overline{\s}$ on $\overline{K}(X)$. Set $\s := \overline{\s}|_{\overline{K}}$. Then $\s\in \Gal(\overline{K}|K)$. Observe that $\overline{w}|_{\overline{K}} = \overline{v} = \overline{w}_1|_{\overline{K}}$. It follows that $\overline{v}c = \overline{w}c = \overline{w}_1 \s c = \overline{v} \s c$ for all $c\in \overline{K}$, that is, 
	\[ \overline{v} = \overline{v}\circ\s \text{ on } \overline{K}. \]
	In other words, $\s$ is a valuation preserving isomorphism of $\overline{K}$ over $K$. By continuity, we can lift $\s$ to a valuation preserving isomorphism $\widehat{\s}$ of $\widehat{\overline{K}}$ over $\widehat{K}$, that is, $\widehat{\s}\in \Gal (\widehat{\overline{K}}|\widehat{K})$ such that
	\begin{align*}
		\widehat{\overline{v}}\circ\widehat{\s} &= \widehat{\overline{v}} \text{ on } \widehat{\overline{K}},\\
		\widehat{\s}|_{\overline{K}} &= \s.
	\end{align*}

\pars We first assume that $(\widehat{K}(X)|\widehat{K},\widehat{w})$ is valuation transcendental with a minimal pair of definition $(a,\g)$ over $\widehat{K}$, where $a\in\overline{K}$. In light of Proposition \ref{Prop completion valn tr induced} and Proposition \ref{Prop valn tr induced unique pair two cases}, we observe that $(K(X)|K,w)$ is valuation transcendental with $(a,\g)$ as a minimal pair of definition for $\overline{w}$ over $K$. Since the extension induced by a valuation transcendental extension is again valuation transcendental, it follows that $(\widehat{K}(X)|\widehat{K},\widehat{w}_1)$ is also valuation transcendental. Observe that $\overline{w}(X-z) = \overline{w}_1 (X-\s z)$ for all $z\in\overline{K}$. Consequently, $\overline{w}(X-\overline{K}) = \overline{w}_1 (X-\overline{K})$. Now,
\[ \g = \overline{w}(X-a) = \overline{w}_1 (X-\s a). \]
It follows that $\g = \max \overline{w}_1 (X-\overline{K})$ and hence $(\s a,\g)$ is a pair of definition for $\overline{w}_1$ over $K$. Since $\overline{w}$ and $\overline{w}_1$ are extensions of $w$, it now follows from [\ref{Dutta min fields implicit const fields}, Remark 3.14] that $(\s a,\g)$ is a minimal pair of definition for $\overline{w}_1$ over $K$. By Corollary \ref{Coro CSKP valn tr}, $(\s a,\g)$ is a minimal pair of definition for $\widehat{w}_1$ over $\widehat{K}$. Note that $a$ and $a_1 := \s a$ are conjugates over $K$. Further, the fact that $a$ is algebraic over $\widehat{K}$ implies that $a_1 = \s a = \widehat{\s} a$ is also a conjugate of $a$ over $\widehat{K}$. 

\pars We now assume that $(\widehat{K}(X)|\widehat{K},\widehat{w})$ is valuation transcendental such that every minimal pair of definition $(a,\g)$ over $\widehat{K}$ satisfies $a\in\overline{\widehat{K}} \setminus \overline{K}$. In light of Proposition \ref{Prop completion valn tr induced}, we obtain that $(a,\g)$ is the unique pair of definition for $\widehat{w}$ over $\widehat{K}$. It then follows from Proposition \ref{Prop valn tr induced unique pair two cases} that $(K(X)|K,w)$ is valuation algebraic. Take a cofinal set $\{\g_\nu\}_{\nu<\l} \subseteq \overline{w}(X-\overline{K})$ and $\{a_\nu\}_{\nu<\l} \subseteq \overline{K}$ such that $\overline{w}(X-a_\nu) = \g_\nu$. Then $\{a_\nu\}_{\nu<\l}$ is a Cauchy sequence in $(\overline{K},\overline{v})$ and has $a$ as the unique limit in $\widehat{\overline{K}}$. The observation $\overline{w}(X-\overline{K}) = \overline{w}_1 (X-\overline{K})$ implies that $\{\g_\nu\}_{\nu<\l}$ is cofinal in $\overline{w}_1 (X-\overline{K})$. Further, the fact that $\widehat{\overline{v}}\circ\widehat{\s} = \widehat{\overline{v}}$ implies that 
\[ \g_\nu = \widehat{\overline{v}}(a-a_\nu) = \widehat{\overline{v}}(\widehat{\s} a-\s a_\nu) \text{ for all } \nu<\l.  \]     
Thus $\{\s a_\nu\}_{\nu<\l}$ is a Cauchy sequence in $(\overline{K},\overline{v})$ with the unique limit $a_1:= \widehat{\s}a$ in $\widehat{\overline{K}}$. The fact that $a\in\overline{\widehat{K}}$ then implies that $a_1\in \overline{\widehat{K}}$. By construction of the induced valuation, we observe that $(\widehat{K}(X)|\widehat{K},\widehat{w}_1)$ is valuation transcendental with the unique pair of definition $(a_1,\g)$.

\pars Finally, we assume that $(\widehat{K}(X)|\widehat{K},\widehat{w})$ is valuation algebraic. It follows from Proposition \ref{Prop valn alg induced} that $(K(X)|K,w)$ is valuation algebraic. Take sequences $\{\g_\nu\}_{\nu<\l}$ and $\{a_\nu\}_{\nu<\l}$ as above. In light of Theorem \ref{Thm CSKP valn alg}, $\{a_\nu\}_{\nu<\l}$ is a pCs of transcendental type in $(\overline{\widehat{K}},\overline{\widehat{v}})$. Consider the sequence $\{\s a_\nu\}_{\nu<\l}$ and suppose that it has the limit $b\in \overline{\widehat{K}}$. The fact that $\widehat{\overline{v}}\circ\widehat{\s} = \widehat{\overline{v}}$ then implies that $\widehat{\s}^{-1} b \in \overline{\widehat{K}}$ is a limit of $\{a_\nu\}_{\nu<\l}$, thereby contradicting the fact that $\{a_\nu\}_{\nu<\l}$ is a pCs of transcendental type in $(\overline{\widehat{K}},\overline{\widehat{v}})$. As a consequence, $\{\s a_\nu\}_{\nu<\l}$ is also a pCs of transcendental type in $(\overline{\widehat{K}},\overline{\widehat{v}})$. By construction, $ (\widehat{K}(X)|\widehat{K},\widehat{w}_1)$ is valuation algebraic.   

\pars The proof of the theorem is now completed in view of the symmetry of the preceding arguments.
\end{proof}


\section{Examples}

    Take the valued field $K:= k(t)$ equipped with the $t$-adic valuation $v:= v_t$. Then $vK=\ZZ$, $Kv=k$ and $\widehat{K} = k((t))$. Fix an extension of $v$ to $\widehat{\overline{K}}$. Take the ordered abelian group $(\ZZ\dirsum\QQ)_{\text{lex}}$ equipped with the lexicographic order. Embed $v\overline{K}=\QQ$ in $(\ZZ\dirsum\QQ)_{\text{lex}}$ by setting $\a\mapsto (0,\a)$ for all $\a\in v\overline{K}$. Set $\g:= (1,0)$. Then $\g > v\overline{K}$.
	
\begin{Example}
	Assume that $\ch k=p>0$. Set 
	\[  a:= \sum_{n=0}^{\infty} t^{p^n}. \]
	Then $a\in\widehat{K}\setminus K$. Observe that $a$ satisfies the Artin-Schreier polynomial $X^p-X+t \in K[X]$. Hence $a\in\overline{K}$. An Artin-Schreier polynomial either splits completely or is irreducible. The fact that $a\notin K$ implies that $X^p-X+t$ is the minimal polynomial of $a$ over $K$.
	
	\pars Take the value transcendental extension of $v$ to $K(X)$ which admits $(a,\g)$ as a pair of definition. The fact that $\g>vK$ implies that $(a,\g)$ is the unique pair of definition. For each $m\in\NN$, set $a_m:= \sum_{n=0}^{m} t^{p^n}$. Then $Q_m:= X-a_m \in K[X]$ is a monic linear polynomial over $K$ and hence is a key polynomial for $v$ over $K$. Observe that \[  \d(Q_m) = v(X-a_m) = v(a-a_m) = v(\sum_{n=m+1}^{\infty} t^{p^n}) = p^{m+1}. \]
	Hence $\{ \d(Q_m) \}_{m\in\NN}$ is cofinal in $vK$ and consequently in $v\overline{K}$. It now follows from Proposition \ref{Prop vf = v_Q f iff d(f) leq d(Q)} that for any polynomial $f(X)\in K[X]$ with $\deg f < p$, we have $vf = v_{Q_m}f$ for some $m\in\NN$. As a consequence, $\{ X-a_m \}_{m\in\NN} \Union \{ X^p-X+t\}$ forms a complete sequence of key polynomials for $v$ over $K$. We then obtain from Theorem \ref{Thm CSKP unique pair} that $\{ X-a_m \}_{m\in\NN} \Union \{ X-a\}$ forms a complete sequence of key polynomials for $v$ over $\widehat{K}$.
\end{Example}

\begin{Example}
	Assume that $\ch k=0$. Set
	\[ a:= \sum_{n=0}^{\infty} \frac{t^n}{n!}. \]
	Then $a\in\widehat{K}$ is transcendental over $K$. Take the value transcendental extension of $v$ to $\widehat{K}(X)$ which admits $(a,\g)$ as a pair of definition. The fact that $\g>v\widehat{K}$ implies that $(a,\g)$ is the unique pair of definition. It follows from Proposition \ref{Prop valn tr induced unique pair two cases} that $(\widehat{K}(X)|\widehat{K},v)$ is induced by the valuation algebraic extension $(K(X)|K,v)$. Hence $(K(X)|K,v)$ is valuation algebraic of type II.
	
	\pars Set $a_m:= \sum_{n=0}^{m} \frac{t^n}{n!}$ and $Q_m:= X-a_m \in K[X]$. Then $\d(Q_m) = m+1$ and hence $\{\d(Q_m)\}_{m\in\NN}$ is cofinal in $v\overline{K}$. It follows from Proposition \ref{Prop vf = v_Q f iff d(f) leq d(Q)}  that $\{X-a_m\}_{m\in\NN}$ forms a complete sequence of key polynomials for $v$ over $K$. We can then conclude from Theorem \ref{Thm CSKP valn alg} that $\{X-a_m\}_{m\in\NN}\Union \{ X-a \}$ forms a complete sequence of key polynomials for $v$ over $\widehat{K}$.  
\end{Example}

\begin{Example}
	Assume that $k$ is algebraically closed and $\ch k=0$. Then the algebraic closure of $\widehat{K}$ is given by [\ref{Kuh vln model}, Theorem 10.2]:
	\[  \overline{\widehat{K}} = P(k) := \Union_{n\in\NN} k((t^{1/n})). \]
	Observe that $P(k)$ is a subfield of the generalized power series field $k((t^\QQ))$ equipped with the $t$-adic valuation. Take distinct primes $p,q$ with $p<q$. Set
	\[ a:= \sum_{n=0}^{\infty} t^{q^n/p^n}.  \]
	Set $a_m:= \sum_{n=0}^{m} t^{q^n/p^n}$ for each $m\in\NN$. Note that $t^{q^n/p^n} \in \overline{K}$ for each $n$ as it satisfies the polynomial $X^{p^n} - t^{q^n} \in K[X]$. It follows that $a_m \in \overline{K}$ for each $m\in\NN$. Now, 
	\[ \g_m:= v(a_m - a_{m+1}) = \frac{q^{m+1}}{p^{m+1}}. \]
	The fact that $p<q$ implies that $\{ \g_m \}_{m\in\NN}$ is cofinal in $v\overline{K}$. Thus $\{ a_m \}_{m\in\NN}$ is a Cauchy sequence in $(\overline{K},v)$. Consequently, $a$ is the unique limit of $\{a_m\}_{m\in\NN}$ in $\widehat{\overline{K}}$. Since the exponents of $t$ in the expression of $a$ do not share a common denominator, it follows that $a\notin P(k)$. Thus
	\[ a\in\widehat{\overline{K}}\setminus\overline{\widehat{K}}. \] 
	Take the value transcendental extension of $v$ to $\widehat{\overline{K}}(X)$ which admits $(a,\g)$ as a pair of definition. The fact that $\g>v\overline{K}= v\widehat{\overline{K}}$ implies that $(a,\g)$ is the unique pair of definition. The fact that $a\notin \overline{\widehat{K}}$ implies that both $(K(X)|K,v)$ and $(\widehat{K}(X)|\widehat{K},v)$ are valuation algebraic extensions. In this case, $(K(X)|K,v)$ is valuation algebraic of type I.
\end{Example}


\section{Applications to ramification theory}\label{Sect ram theory}

\begin{Lemma}\label{Lemma d(K(a)|K) = d(K^c(a)|K^c)}
	Let $(K,v)$ be a henselian valued field and fix an extension of $v$ to $\overline{\widehat{K}}$. Take any $a\in K^\sep$. Then $K(a)$ and $\widehat{K}$ are linearly disjoint over $K$. Further, $d(K(a)|K,v) = d(\widehat{K}(a)|\widehat{K},v)$.
	\newline Additionally, take any $b$ separable-algebraic over $\widehat{K}$. Then there exists $a\in K^\sep$ such that $\widehat{K}(b) = \widehat{K}(a)$.
\end{Lemma}

\begin{proof}
	The fact that $(K,v)$ is henselian implies that $(\widehat{K},v)$ is also henselian [\ref{Warner topo fields}, Theorem 32.19]. Take $\g\in v\overline{K}$ such that $\g> \kras(a,K)$. Take the extension of $v$ to $\overline{K}(X)$ which has $(a,\g)$ as a pair of definition. Take any other pair of definition $(a^\prime,\g)$. Then $v(a-a^\prime) \geq \g > \kras (a,K)$. It follows from Krasner's Lemma that $a\in K(a^\prime)$. Consequently, $(a,\g)$ is a minimal pair of definition for $v$ over $K$. Take the induced extension $(\widehat{K}(X)|\widehat{K},v)$. It follows from Theorem \ref{Thm CSKP valn tr cofinal} that $(a,\g)$ is also a minimal pair of definition for $v$ over $\widehat{K}$ and that $[K(a):K] = [\widehat{K}(a):\widehat{K}]$. Observe that $\widehat{K}(a) = \widehat{K(a)}$ and hence $(\widehat{K}(a)|K(a),v)$ is an immediate extension. We can then conclude from the Lemma of Ostrowski that
	\[ d(K(a)|K,v) = d(\widehat{K}(a)|\widehat{K},v). \]
	Now take some $b$ which is separable-algebraic over $\widehat{K}$. Take $\d\in v\overline{\widehat{K}} = v\overline{K}$ such that $\d> \kras(b,\widehat{K})$ and consider the valuation transcendental extension of $v$ to $\overline{\widehat{K}}(X)$ which admits $(b,\d)$ as a pair of definition. Then $(b,\d)$ is a minimal pair of definition for $v$ over $\widehat{K}$. The fact that $\d\in v\overline{K}$ implies that $(\widehat{K}(X)|\widehat{K},v)$ is residue transcendental. In light of Proposition \ref{Prop completion valn tr induced}, we can choose $a \in K^\sep$ such that $(a,\d)$ is a minimal pair of definition for $v$ over $K$ and also over $\widehat{K}$. The fact that both $(a,\d)$ and $(b,\d)$ are minimal pairs of definition for $v$ over $\widehat{K}$ implies that $[\widehat{K}(a):\widehat{K}] = [\widehat{K}(b):\widehat{K}]$. Moreover, we have $v(a-b) \geq \d > \kras(b,\widehat{K})$ and hence $b\in \widehat{K}(a)$ by Krasner's Lemma. We can then conclude that $\widehat{K}(a) = \widehat{K}(b)$.
\end{proof}

The following corollary is immediate:

\begin{Corollary}\label{Coro d(K(a)|K) = d(K^c(a)|K^c)}
	Let assumptions be as in Lemma \ref{Lemma d(K(a)|K) = d(K^c(a)|K^c)}. Take a henselian valued field $(L,v)$ such that $K\subseteq L \subseteq \widehat{K}$. Then the conclusions of Lemma \ref{Lemma d(K(a)|K) = d(K^c(a)|K^c)} hold true with $\widehat{K}$ replaced by $L$.
\end{Corollary}

\begin{proof}
	The assumption that $K\subseteq L \subseteq\widehat{K}$ implies that $\widehat{K}=\widehat{L}$. The assertions now follow directly from Lemma \ref{Lemma d(K(a)|K) = d(K^c(a)|K^c)}.
\end{proof}

\pars We can now give a \textbf{proof of Theorem \ref{Thm K^c properties}}:

\begin{proof}
	We first assume that $(K,v)$ is separably defectless. Take some $b\in\overline{\widehat{K}}$ which is separable over $\widehat{K}$. In light of Lemma \ref{Lemma d(K(a)|K) = d(K^c(a)|K^c)}, there exists $a\in K^\sep$ such that $\widehat{K}(b) = \widehat{K}(a)$ and hence 
	\[ d(\widehat{K}(b)|\widehat{K},v) = d(\widehat{K}(a)|\widehat{K},v) = d(K(a)|K,v) = 1. \]
	Since this holds for any arbitrary $b$ which is separable-algebraic over $\widehat{K}$, we conclude from the primitive element theorem that $(\widehat{K},v)$ is separably defectless. 
	\pars Conversely, assume that $(\widehat{K},v)$ is separably defectless. Take any $a$ which is separable-algebraic over $K$. Then $a$ is also separable-algebraic over $\widehat{K}$. It now follows from Lemma \ref{Lemma d(K(a)|K) = d(K^c(a)|K^c)} that
	\[ d(K(a)|K,v) = d(\widehat{K}(a)|\widehat{K},v) = 1. \]
	The final equivalence of $(i)$ follows from [\ref{Kuh A-S extensions and defectless fields paper}, Theorem 5.1].
	
	\parm Recall that a henselian valued field $(K,v)$ is (separably) tame if and only if $(K,v)$ is (separable-algebraically) algebraically maximal, $vK$ is $p$-divisible and $Kv$ is perfect, where $p$ is the characteristic exponent [\ref{Kuh vln model}, Lemma 11.8]. Assume that $(K,v)$ is separably tame. Then $(K,v)$ is separably defectless, hence by $(i)$, $(\widehat{K},v)$ is defectless. In particular, $(\widehat{K},v)$ is algebraically maximal. Noting that the completion is an immediate extension, we conclude that $(\widehat{K},v)$ is tame. Conversely, assume that $(\widehat{K},v)$ is tame. Then $(\widehat{K},v)$ is defectless and hence $(K,v)$ is separably defectless. Consequently, $(K,v)$ is separable-algebraically maximal. It follows that $(K,v)$ is separably tame. The remaining equivalence of $(ii)$ follows analogously. 
	
	\parm We now assume that $(K,v)$ is separable-algebraically maximal. Suppose that $(\widehat{K},v)$ is not separable-algebraically maximal. Then there exists a nontrivial finite separable immediate extension $(\widehat{K}(a)|\widehat{K},v)$. Observe that $[\widehat{K}(a):\widehat{K}] = d(\widehat{K}(a)|\widehat{K},v)$ by the Lemma of Ostrowski. In light of Lemma \ref{Lemma d(K(a)|K) = d(K^c(a)|K^c)}, we can assume that $a\in K^\sep$. Thus $d(K(a)|K,v) = [K(a):K]$ and as a consequence, $(K(a)|K,v)$ is a nontrivial separable immediate extension, thereby yielding a contradiction. It follows that $(\widehat{K},v)$ is separable-algebraically maximal. The reverse direction of $(iii)$ is an immediate corollary of Lemma \ref{Lemma d(K(a)|K) = d(K^c(a)|K^c)}.
	
	\parm The first equivalence of $(iv)$ follows easily from Lemma \ref{Lemma d(K(a)|K) = d(K^c(a)|K^c)}. The second equivalence is a direct corollary of [\ref{Warner topo fields}, Theorem 30.27].
	
	\parm We now assume that $(K,v)$ is algebraically maximal. Suppose that $(\widehat{K},v)$ is not algebraically maximal. Then there exists a nontrivial finite immediate extension $(L|\widehat{K},v)$. Set $L^\prime$ to be the separable closure of $\widehat{K}$ in $L$. Then $(L^\prime|\widehat{K},v)$ is a finite separable immediate extension and $(L|L^\prime,v)$ is a purely inseparable immediate extension. It follows from $(iii)$ that $(\widehat{K},v)$ is separable-algebraically maximal. Consequently, $L^\prime = \widehat{K}$, that is, $(L|\widehat{K})$ is a nontrivial finite purely inseparable immediate extension. Since any finite purely inseparable extension is a tower of purely inseparable extensions of degree $p$, we obtain some $\eta\in L$ such that 
	\[ (\widehat{K}(\eta)|\widehat{K},v) \text{ is a purely inseparable immediate extension of degree }p. \] 
	Denote by $\dist(\eta,\widehat{K})$ the least initial segment of $v\overline{\widehat{K}}$ containing $v(\eta-\widehat{K})$. The facts that $\eta\notin\widehat{K}$ and $(\widehat{K},v)$ is complete imply that there exists $\a\in vK$ such that $\dist(\eta,\widehat{K}) < \a$. Take the minimal polynomial $X^p - a \in \widehat{K}[X]$ of $\eta$ over $\widehat{K}$. Invoking the theorem of the continuity of roots, we observe that there exists $a^\prime\in K$ and $\d\in vK$ such that 
	\[ v(\eta-\eta^\prime) > \a \text{ whenever } v(a-a^\prime) > \d,  \]
	where $(\eta^\prime)^p  = a^\prime$. Thus $\eta^\prime$ is purely inseparable over $K$ with $[K(\eta^\prime):K] \leq p$. It follows from [\ref{Kuh A-S extensions and defectless fields paper}, Lemma 2.17] that 
	\[ \dist(\eta,\widehat{K}) = \dist(\eta^\prime,\widehat{K}). \]
	As a consequence, $\dist(\eta^\prime,\widehat{K}) < \a$ and hence $\eta^\prime\notin\widehat{K}$. Thus $\eta^\prime$ is purely inseparable over $\widehat{K}$. Moreover, 
	\[ [K(\eta^\prime):K] = p = [\widehat{K}(\eta^\prime):\widehat{K}]. \]
	Recall that $(\widehat{K}(\eta)|\widehat{K},v)$ is an immediate extension. It now follows from [\ref{Kuh A-S extensions and defectless fields paper}, Lemma 2.21] that $(\widehat{K}(\eta^\prime)|\widehat{K},v)$ is also immediate. The observation that $\widehat{K}(\eta^\prime) = \widehat{K(\eta^\prime)}$ now implies that
	\[ (K(\eta^\prime)|K,v) \text{ is a purely inseparable immediate extension of degree }p. \]
	But this contradicts the assumption that $(K,v)$ is algebraically maximal. It follows that $(\widehat{K},v)$ is also algebraically maximal.
\end{proof}

We now provide an example to show that the converse to Theorem \ref{Thm K^c properties}$(v)$ does not hold true. The first part of the construction follows an example due to F. K. Schmidt.

\begin{Example}
	Take a non-zero prime number $p$ and consider the valued field $\FF_p((t))$ equipped with the $t$-adic valuation $v:= v_t$. Fix an extension of $v$ to $\overline{\FF_p((t))}$. Observe that $\FF_p((t))|\FF_p(t)$ has infinite transcendence degree. Choose $\zeta\in\FF_p((t))$ which is transcendental over $\FF_p(t)$. Set $k:= \FF_p(t,\zeta^p)$ and $K:= k^h$. The fact that the henselization is a separable-algebraic extension implies that $K|k$ is linearly disjoint from the purely inseparable extension $\FF_p (t,\zeta)|k$. Consequently, $[K(\zeta):K] = [\FF_p(t,\zeta):k] = p$. We now have a chain of inclusions
	\[ \FF_p(t) \subsetneq K \subsetneq K(\zeta) \subsetneq \FF_p((t)). \]
	Since $(\FF_p((t))|\FF_p(t),v)$ is an immediate extension, the subextension $(K(\zeta)|K,v)$ is also immediate. Consequently, $(K,v)$ is not algebraically maximal. 
	\pars The chain $\FF_p(t) \subseteq K \subseteq \FF_p((t)) = \widehat{\FF_p(t)}$ implies that $\widehat{K} = \FF_p((t))$. Now $\FF_p((t))$ being a power series field, is maximal. In particular, $\FF_p((t))$ is algebraically maximal and defectless.
\end{Example}
The preceding example also illustrates that $(\widehat{K},v)$ being defectless does not necessarily imply that $(K,v)$ is defectless.

\section{Density of $K(X)$ in $\widehat{K}(X)$}\label{Sect density}

\begin{Proposition}\label{Prop not density of K(X)}
	Let $(\overline{\widehat{K}}(X)|K,v)$ be an extension of valued fields. Assume that $(K(X)|K,v)$ is either value transcendental with a unique pair of definition or is valuation algebraic of type II. Then $\widehat{K}$ is not contained in $\widehat{K(X)}$.
\end{Proposition}

\begin{proof}
	We first assume that $(K(X)|K,v)$ is valuation algebraic of type II. Suppose that $\widehat{K}$ is contained in $\widehat{K(X)}$. Then $\widehat{K}(X)\subseteq \widehat{K(X)}$ and hence $v\widehat{K}(X) \subseteq v\widehat{K(X)}= vK(X)$. The fact that $(K(X)|K,v)$ is valuation algebraic implies that $vK(X)/vK$ is a torsion group. It follows that $v\widehat{K}(X)/vK$ is also a torsion group. However, this yields a contradiction, as the fact that $(\widehat{K}(X)|\widehat{K},v)$ is value transcendental implies that $\rr v\widehat{K}(X)/v\widehat{K} = 1$, that is, $\rr v\widehat{K}(X)/vK=1$. 
	
	\pars We now assume that $(K(X)|K,v)$ admits a unique pair of definition $(a,\g)$ for some $a\in\overline{K}$. Then $\g>v\overline{K}$. Suppose that $\widehat{K}\subseteq \widehat{K(X)}$. Then $\widehat{K(a)} = \widehat{K}(a) \subseteq \widehat{K(X)}(a)=\widehat{K(a,X)}$. So without any loss of generality we can assume that $a\in K$. Consequently, $\g = v(X-a) \in vK(X)$. Take some $b\in\widehat{K}\setminus K$. The assumption $\widehat{K}\subseteq \widehat{K(X)}$ implies that there exist coprime polynomials $f(X),g(X) \in K[X]$ such that 
	 \[ v(\frac{f(X)}{g(X)}-b) > \g. \]
	 The fact that $(a,\g)$ is the unique pair of definition for $v$ over $K$ implies that $vg\in v\overline{K}$ whenever $g(a)\neq 0$. Otherwise $vg > v\overline{K}$. In either case, we can conclude that
	 \[ v(f(X)-bg(X)) > \g+vg > v\overline{K} = v\overline{\widehat{K}}.  \]
	 Set $h(X): = f(X)-bg(X) \in \widehat{K}[X]$. Note that $(a,\g)$ is also the unique pair of definition for $v$ over $\widehat{K}$. Then the condition $vh>v\overline{\widehat{K}}$ necessarily implies that $h(a)=0$, that is,
	 \begin{equation}\label{eqn}
	 	f(a) = bg(a).
	 \end{equation} 
	 The conditions $a\in K$, $f(X),g(X)\in K[X]$ imply that $f(a),g(a) \in K$. The coprimality of $f(X)$ and $g(X)$ implies that $f(a)$ and $g(a)$ can not vanish simultaneously. But then (\ref{eqn}) contradicts the assumption that $b\in \widehat{K}\setminus K$. Hence $\widehat{K}$ is not contained in $\widehat{K(X)}$. 
\end{proof}

We provide now a \textbf{proof of Theorem \ref{Thm density of K(X)}}:

\begin{proof}
	We first assume that $(K(X)|K,v)$ is valuation transcendental such that $vK$ is cofinal in $vK(X)$. Take a pair of definition $(a,\g)$ for $v$ over $K$. Then $(a,\g)$ is also a pair of definition for the induced extension $(\widehat{K}(X)|\widehat{K},v)$. We observe from Corollary \ref{Coro CSKP valn tr} that the extension $(\widehat{K}(X)|K(X),v)$ is immediate. Consequently, $\a \in vK(X)$. Take polynomials $f(X),g(X) \in \widehat{K}[X]$. Write
	\[ f(X) = \sum_{i=0}^{n} c_i X^i = \sum_{i=0}^{n} C_i (X-a)^i.\]
	The coefficients $C_i$ arising from the Hasse derivative are given by 
	\[ C_i = c_i + \binom{i+1}{i} a c_{i+1} + \dotsc + \binom{n}{i} a^{n-i} c_n. \]
	Similarly, we can write $g(X)$ as 
	\[ g(X) = \sum_{i=0}^{m} d_i X^i = \sum_{i=0}^{m} D_i (X-a)^i, \]
	where
	\[ D_i = d_i + \binom{i+1}{i} a d_{i+1} + \dotsc + \binom{m}{i} a^{m-i} d_m. \]
	Observe that $C_i, D_j \in \overline{\widehat{K}}$ and hence $v C_i, vD_j \in v \overline{K}$. Set
	\[ C:= \min \{v C_i\} \text{ and } D:= \min \{ v D_j \}. \]  
	Take $\b\in vK$ such that
	\begin{align*}
		& \b + i\g > \a \text{ for all } i=0, \dotsc, \max \{m,n\},\\
		& \b + \min\{ C,D\} + k\g > \a + 2 vg \text{ for all } k=0, \dotsc, m+n.
	\end{align*}
	Note that the choice of such a $\b$ is made possible by the facts that $\a \in vK(X)$ and $vK$ is cofinal in $vK(X)$. 
	
	\pars We can again express $f$ as $f(X) = c_n \prod_{ i=1}^{n} (X-z_i)$. By the continuity of roots, there exists $\d \in vK$ such that after suitable enumeration of the roots, we obtain $v(z_i-z^\prime_i) > \d(f)$ for all $i$ whenever $v(c_i -c^\prime_i) > \d$ for all $i$, where $\sum_{i=0}^{n} c^\prime_i X^i = c^\prime_n \prod_{i=1}^{n} (X-z^\prime_i)$. The fact that $c_i \in \widehat{K}$ implies that we can take $c^\prime_i \in K$ such that 
	\begin{equation}\label{eqn v(c-c prime)}
		v(c_i - c^\prime_i) > \max \{ \d, \b, \b - va, \dotsc, \b - nva \}.
	\end{equation}
	 Take $f^\prime(X):= \sum_{i=0}^{n} c^\prime_i X^i \in K[X]$. Employing exactly the same arguments as in the proof of Lemma \ref{Lemma deg f and d(f)}, we obtain that $vf = vf^\prime$ and $\deg f = \deg f^\prime$. Consider the expansion $f^\prime = \sum_{i=0}^{n} C^\prime_i (X-a)^i$. Then,
	 \[ C_i - C^\prime_i = (c_i - c^\prime_i) + \binom{i+1}{i} a (c_{i+1}-c^\prime_{i+1}) + \dotsc + \binom{n}{i} a^{n-i} (c_n - c^\prime_n). \] 
	 For any non-zero term in the above expansion, we obtain from (\ref{eqn v(c-c prime)}) that
	 \[ v\binom{j}{i} a^{j-i} (c_j - c^\prime_j) = (j-i)va + v(c_j-c^\prime_j) > (j-i)va+ \b-(j-i)va = \b. \]
	 As a consequence of the triangle inequality, we now have that
	 \[ v(C_i - C^\prime_i) > \b \text{ for all }i. \]
	 Consider the expansion
	 \[ f-f^\prime = \sum_{i=0}^{n} (C_i - C^\prime_i)(X-a)^i. \]
	 The fact that $(a,\g)$ is a pair of definition for $v$ implies that $v(f-f^\prime) = \min \{ v(C_i-C^\prime_i)+ i\g \}$. It then follows from our preceding discussions that $v(f-f^\prime)> \b+i\g$ for some $i\in\{0,\dotsc,n\}$. From our choice of $\b$, we conclude that
	 \[ v(f-f^\prime)>\a. \]
	 Analogously, there exists $g^\prime(X) \in K[X]$ given by $g^\prime = \sum_{j=0}^{m} d^\prime_j X^j = \sum_{j=0}^{m} D^\prime_j (X-a)^j$ satisfying the following conditions:
	 \begin{align*}
	 	& \deg g = \deg g^\prime, \, vg = vg^\prime, \, v(D_j - D^\prime_j) > \b \text{ for all }j, \\
	 	& v(g-g^\prime)>\a.
	 \end{align*}
	 We now consider $v(\frac{f}{g} - \frac{f^\prime}{g^\prime})$. Observe that
	 \[ v(\frac{f}{g} - \frac{f^\prime}{g^\prime}) = v(fg^\prime-f^\prime g) - v(gg^\prime) = v(fg^\prime-f^\prime g) - 2vg \]
	 as $vg = vg^\prime$. We have the expansion 
	 \[ fg^\prime - f^\prime g = \sum_{k=0}^{m+n} \sum_{i+j=k}(C_i D^\prime_j - C^\prime_iD_j)(X-a)^k. \]
	 Write $C_i D^\prime_j - C^\prime_iD_j = (D^\prime_j - D_j)C_i + (C_i - C^\prime_i) D_j$. The facts that $v(D^\prime_j - D_j) > \b$ and $vC_i \geq C$ imply that $v(D^\prime_j - D_j)C_i > \b+C$. Similarly, $v(C_i - C^\prime_i) D_j > \b+D$. It then follows from the triangle inequality that
	 \[ v\sum_{i+j=k} (C_i D^\prime_j - C^\prime_iD_j) > \b+ \min\{C,D\}. \]
	 The fact that $(a,\g)$ is a pair of definition for $v$ implies that
	 \[ v(fg^\prime - f^\prime g) = \min_{k=0,\dotsc,m+n} \{ v\sum_{i+j=k} (C_i D^\prime_j - C^\prime_iD_j) + k\g \}. \]
	 Consequently, $v(fg^\prime - f^\prime g) > \b + \min\{ C,D \}+ k\g$ for some $k\in\{0,\dotsc,m+n\}$. From our choice of $\b$, we conclude that $v(fg^\prime - f^\prime g) > \a + 2vg$ and as a consequence, 
	 \[ v(\frac{f}{g} - \frac{f^\prime}{g^\prime}) = v(fg^\prime-f^\prime g) - 2vg > \a.\]
	 We have thus proved the assertions of the theorem when $(K(X)|K,v)$ is valuation transcendental such that $vK$ is cofinal in $vK(X)$.
	 
	 \parm We now assume that $(K(X)|K,v)$ is valuation algebraic of type I. We have observed in Theorem \ref{Thm CSKP valn alg} that $(\widehat{K}(X)|K(X),v)$ is immediate. Consequently, $\a\in vK(X)$. Take a complete sequence of key polynomials $\{Q_\nu\}_{\nu\in\L}$ for $v$ over $K$. Then $\{Q_\nu\}_{\nu\in\L}$ also forms a complete sequence of key polynomials for $v$ over $\widehat{K}$ (Theorem \ref{Thm CSKP valn alg}). For each $\nu\in\L$, take a maximal root $a_\nu$ of $Q_\nu$ and set $\g_\nu:= \d(Q_\nu) = v(X-a_\nu)$. Set
	 \[ v_\nu:= v_{a_\nu,\g_\nu}. \]
	 The fact that $\{Q_\nu\}_{\nu\in\L}$ is a complete sequence of key polynomials for $v$ over $K$ implies that $\{\g_\nu\}_{\nu\in\L}$ is cofinal in $v(X-\overline{K})$. Then $(K(X)|K,v)$ being valuation algebraic of type I implies that $\{a_\nu\}_{\nu\in\L}$ is a pCs of transcendental type in $(\overline{\widehat{K}},v)$ with limit $X$. It now follows from the proof of [\ref{Kaplansky}, Theorem 2] that for any polynomial $h(X) \in \widehat{K}[X]$, we have
	 \[ vh = v_\nu h \text{ for all } \nu\in\L  \text{ sufficiently large}.\]
	 Further, the fact that $Q_\nu$ is a key polynomial for $v$ over $K$ with maximal root $a_\nu$ implies that $(a_\nu,\g_\nu)$ is a minimal pair of definition for $v_\nu$ over $K$. It then follows from Corollary \ref{Coro CSKP valn tr} that $(a_\nu,\g_\nu)$ is also a minimal pair of definition for $v_\nu$ over $\widehat{K}$. In light of [\ref{APZ all valns on K(X)}, Theorem 5.1], we then conclude that
	 \begin{equation}\label{eqn vh geq v_nu h}
	 	vh \geq v_\nu h \text{ for all polynomials }h(X)\in\widehat{K}[X], \text{ for all } \nu\in\L.
	 \end{equation}
	 Take $f,g \in \widehat{K}[X]$. Choose $\nu\in\L$ large enough such that $vf=v_\nu f$ and $vg = v_\nu g$. Fix an extension of $v_\nu$ to $\overline{\widehat{K}}(X)$. The fact that $(K(X)|K,v)$ is valuation algebraic implies that $(\overline{K}(X)|\overline{K},v)$ is immediate. Consequently, $\g_\nu \in v\overline{K}$. It follows that $(K(X)|K,v_\nu)$ is residue transcendental. In particular, $vK$ is cofinal in $v_\nu K(X)$. Take $\a_0 \in vK$ such that 
	 \[ \a_0 > \max \{ \a, vf,vg \}. \]
	 From our preceding discussions in the valuation transcendental case, we obtain $f^\prime, g^\prime \in K[X]$ such that
	 \begin{align*}
	 	\deg f = \deg f^\prime, \, \, \deg g = \deg g^\prime,& \, \, v_\nu f  = v_\nu f^\prime, \, \, v_\nu g  = v_\nu g^\prime,\\
	 	v_\nu(f-f^\prime) > \a_0,  \, \,  v_\nu(g-g^\prime) > \a_0, & \, \, v_\nu(\frac{f}{g} - \frac{f^\prime}{g^\prime})  > \a_0.
	 \end{align*}    
 As a consequence of (\ref{eqn vh geq v_nu h}), we then have that
 \begin{align*}
 	&v(f-f^\prime)\geq v_\nu(f-f^\prime) >\a_0 >\a,\\
 	&v(g-g^\prime)\geq v_\nu(g-g^\prime) >\a_0>\a.
 \end{align*}    
The fact that $\a_0 > \max\{ vf,vg \}$ then yields the following in light of the triangle inequality:
\[ vf = vf^\prime \text{ and } vg = vg^\prime. \]
Moreover,
\begin{align*}
	v(\frac{f}{g} - \frac{f^\prime}{g^\prime}) &= v(fg^\prime-f^\prime g) - v(gg^\prime) = v(fg^\prime - f^\prime g) - 2vg\\
	&= v(fg^\prime - f^\prime g) - 2v_\nu g\\
	&\geq v_\nu (fg^\prime - f^\prime g) - 2v_\nu g\\
	&= v_\nu (\frac{f}{g} - \frac{f^\prime}{g^\prime}) >\a_0 >\a.
\end{align*}
We have thus proved the theorem.
\end{proof}


\section{Implicit constant fields}\label{Sect imp cnst fields}

Given an extension of valued fields $(\overline{K(X)}|K,v)$, the implicit constant field of the extension $(K(X)|K,v)$ is defined as 
\[ IC(K(X)|K,v):= \overline{K}\sect K(X)^h. \]
Thus $IC(K(X)|K,v)$ is a separable-algebraic extension of $K^h$. As a consequence, the extension $IC(K(X)|K,v)|K^h$ is simple whenever it is finite. 

\pars We now assume that $(\overline{\widehat{K}(X)}|K,v)$ is an extension of valued fields. Then,
\[ IC(K(X)|K,v) = \overline{K}\sect K(X)^h \subseteq \overline{\widehat{K}} \sect \widehat{K}(X)^h = IC(\widehat{K}(X)|\widehat{K},v). \]
The fact that $\widehat{K}^h$ is a subset of $IC(\widehat{K}(X)|\widehat{K},v)$ then implies the following:
\begin{equation}\label{eqn IC and completion}
	\widehat{K}^h. IC(K(X)|K,v) \subseteq IC(\widehat{K}(X)|\widehat{K},v),
\end{equation}
where $\widehat{K}^h$ denotes the henselization of $\widehat{K}$. We can now give a \textbf{proof of Proposition \ref{Prop IC of K and K hat}}:

\begin{proof}
	We first assume that $(K(X)|K,v)$ is valuation transcendental such that it admits a unique pair of definition $(a,\g)$. Then the induced extension $(\widehat{K}(X)|\widehat{K},v)$ also has $(a,\g)$ as the unique pair of definition. It follows from [\ref{Dutta min fields implicit const fields}, Theorem 1.1] that $K^h\subseteq IC(K(X)|K,v)\subseteq K^h(a)$ and hence $IC(K(X)|K,v)$ is a finite extension of $K^h$. Take $b\in K^\sep$ such that 
	\[ IC(K(X)|K,v) = K^h(b). \]
	It follows from [\ref{Dutta min fields implicit const fields}, Theorem 7.2] that $K^h(a)|K^h(b)$ is a purely inseparable extension. Consequently, the extension $\widehat{K}^h(a)|\widehat{K}^h(b)$ is also purely inseparable [\ref{Lang}, V, \S 6, Proposition 6.5]. Further, the fact that $b$ is separable over $K$ implies that $\widehat{K}^h(b)|\widehat{K}$ is separable. It follows that $\widehat{K}^h(b)$ is the separable-algebraic closure of $\widehat{K}$ in $\widehat{K}^h(a)$. In light of [\ref{Dutta min fields implicit const fields}, Theorem 7.2], we conclude that
	\[ IC(\widehat{K}(X)|\widehat{K},v) = \widehat{K}^h(b). \]

	\pars We now assume that either $(K(X)|K,v)$ is valuation transcendental such that $vK$ is cofinal in $vK(X)$ or $(K(X)|K,v)$ is valuation algebraic of type I. Take any $b\in IC(\widehat{K}(X)|\widehat{K},v)$. Then $b$ is separable-algebraic over $\widehat{K}^h$. We have the inclusions $K^h \subseteq \widehat{K}^h \subseteq \widehat{K^h}$. In light of Corollary \ref{Coro d(K(a)|K) = d(K^c(a)|K^c)}, we can then choose $a$ separable-algebraic over $K^h$ such that $\widehat{K}^h(b) = \widehat{K}^h(a)$. The fact that $\widehat{K}^h \subseteq \widehat{K}(X)^h$ then implies that
	\[ \widehat{K}(X)^h(b) = \widehat{K}(X)^h(a). \]
	By definition, $b\in\widehat{K}(X)^h$. It follows that $a\in \widehat{K}(X)^h$. We have observed in Theorem \ref{Thm density of K(X)} that $K(X)$ lies dense in $\widehat{K}(X)$. Then $K(X) \subseteq \widehat{K}(X) \subseteq \widehat{K(X)}$. Fix an extension of $v$ to $\overline{\widehat{K(X)}}$. Upon taking henselizations, we have
	\[  K(X)^h \subseteq \widehat{K}(X)^h \subseteq \widehat{K(X)}^h \subseteq \widehat{K(X)^h}.  \]
	We then obtain from Corollary \ref{Coro d(K(a)|K) = d(K^c(a)|K^c)} that $[\widehat{K}(X)^h(a):\widehat{K}(X)^h] = [K(X)^h(a):K(X)^h]$. Consequently, $a\in K(X)^h$. The fact that $a\in\overline{K}$ then implies that $a\in IC(K(X)|K,v)$. We have thus obtained that $b\in \widehat{K}^h(a) \subseteq \widehat{K}^h.IC(K(X)|K,v)$ and as a consequence, 
	\[ IC(\widehat{K}(X)|\widehat{K},v) \subseteq \widehat{K}^h. IC(K(X)|K,v). \]
	The proposition now follows from (\ref{eqn IC and completion}).
\end{proof}

\begin{Remark}
	In the context of Proposition \ref{Prop IC of K and K hat} we obtain that 
	\[ IC(K(X)|K,v) \text{ is dense in } IC(\widehat{K}(X)|\widehat{K},v). \]
	Indeed, we have the inclusion $\widehat{K}^h \subseteq \widehat{K^h}$. The fact that $IC(K(X)|K,v)|K^h$ is an algebraic extension implies that $\widehat{K^h}$ is contained in the completion of $IC(K(X)|K,v)$. Consequently, $\widehat{K}^h.IC(K(X)|K,v)$ is contained in the completion of $IC(K(X)|K,v)$.
\end{Remark}

\begin{Corollary}
	Let notations and assumptions be as in Proposition \ref{Prop IC of K and K hat}. Assume further that $(K,v)$ is henselian and $IC(K(X)|K,v)|K$ is a finite extension. Then
	\[  IC(\widehat{K}(X)|\widehat{K},v) \text{ is the completion of } IC(K(X)|K,v). \]
\end{Corollary}

Since we do not have a precise description of the implicit constant field for valuation algebraic extensions, the problem appears to be more complicated when $(K(X)|K,v)$ is valuation algebraic of type II. However, we have a satisfactory answer when the henselization of $K$ is separably tame.

\begin{Proposition}
	Let $(\overline{\widehat{K}(X)}|K,v)$ be an extension of valued fields. Assume that $(K(X)|K,v)$ is valuation algebraic of type II. Further, assume that $K^h$ is separably tame. Then,
	\[ \widehat{K}^h. IC(K(X)|K,v) = IC(\widehat{K}(X)|\widehat{K},v).  \]
\end{Proposition}

\begin{proof}
	We have observed in [\ref{Dutta imp const fields key pols valn alg extns}, Section 5] that we can construct sequences $\{\g_\nu\}_{\nu<\l} \subseteq v(X-\overline{K})$ and $\{a_\nu\}_{\nu<\l} \subseteq \overline{K}$ such that $\{\g_\nu\}_{\nu<\l}$ is cofinal in $v(X-\overline{K})$, $v(X-a_\nu) = \g_\nu$ and $(a_\nu,\g_\nu)$ is a minimal pair of definition for $v_{a_\nu,\g_\nu}$ over $K$ for all $\nu<\l$. $(K(X)|K,v)$ being a valuation algebraic extension implies that $(\overline{K}(X)|\overline{K},v)$ is immediate and hence $\g_\nu\in v\overline{K}$. In light of [\ref{Dutta min fields implicit const fields}, Proposition 3.6], we can thus further assume that $a_\nu$ is separable-algebraic over $K$. The assumption that $K^h$ is separably tame implies that $K^\sep = K^r$. It now follows from [\ref{Dutta imp const fields key pols valn alg extns}, Theorem 6.2] that
	\[ IC(K(X)|K,v) = K^h(a_\nu\mid \nu<\l). \]
	It follows from Theorem \ref{Thm CSKP valn alg} that $\{a_\nu\}_{\nu<\l}$ is a Cauchy sequence in $(\overline{K},v)$ and the unique limit $a\in\widehat{\overline{K}}$ is algebraic over $\widehat{K}$. Moreover, $(a,\g)$ is the unique pair of definition for the induced extension $(\widehat{K}(X)|\widehat{K},v)$. It then follows from [\ref{Dutta min fields implicit const fields}, Theorem 7.2] that $IC(\widehat{K}(X)|\widehat{K},v)$ is the separable-algebraic closure of $\widehat{K}^h$ in $\widehat{K}^h(a)$. Therefore, $IC(\widehat{K}(X)|\widehat{K},v) = \widehat{K}^h$ whenever $a$ is purely inseparable over $\widehat{K}^h$. The inclusions $\widehat{K}^h \subseteq \widehat{K}^h.IC(K(X)|K,v) \subseteq IC(\widehat{K}(X)|\widehat{K},v) = \widehat{K}^h$ then imply that all the containments are actually equalities. We have thus proved the assertion in this case.
	
	\pars Otherwise, $a$ is not purely inseparable over $\widehat{K}^h$. So $\kras(a,\widehat{K}^h)$ is well-defined. Note that $\kras(a,\widehat{K}^h) \in v\overline{\widehat{K}} = v\overline{K}$. As a consequence of $\{a_\nu\}_{\nu<\l}$ being a Cauchy sequence in $(\overline{K},v)$, we can take some $a_\nu$ such that $v(a-a_\nu) = \g_\nu > \kras(a,\widehat{K}^h)$. It now follows from Krasner's Lemma that $\widehat{K}^h(a,a_\nu)|\widehat{K}^h(a_\nu)$ is purely inseparable. Consequently, $\widehat{K}^h(a,a_\nu\mid \nu<\l)|\widehat{K}^h(a_\nu\mid \nu<\l)$ is purely inseparable. The relations $K^h(a_\nu\mid \nu<\l) = IC(K(X)|K,v) \subseteq IC(\widehat{K}(X)|\widehat{K},v) \subseteq \widehat{K}^h(a)$ then imply that
	\[ \widehat{K}^h(a)|\widehat{K}^h(a_\nu\mid \nu<\l) \text{ is purely inseparable}.  \]
	Since we chose $a_\nu$ to be separable-algebraic over $K$, the extension $\widehat{K}^h(a_\nu\mid \nu<\l)|\widehat{K}^h$ is also separable-algebraic. It now follows from [\ref{Dutta min fields implicit const fields}, Theorem 7.2] that
	\[ IC(\widehat{K}(X)|\widehat{K},v) = \widehat{K}^h(a_\nu\mid \nu<\l). \]
	We have thus proved the proposition. 
\end{proof}



\end{document}